\documentclass[11pt]{article}
\usepackage{latexsym,amsfonts,amssymb,amsmath,amsthm}
\usepackage{graphicx}

\usepackage[usenames,dvipsnames]{color}
\usepackage{ulem}
\usepackage{comment}
\usepackage[multiple]{footmisc}

\newtheorem{thm}{Theorem}[section]
\newtheorem{cor}[thm]{Corollary}
\newtheorem{lem}[thm]{Lemma}

\newtheorem{defn}[thm]{Definition}
\newtheorem{rem}[thm]{Remark}

\numberwithin{equation}{section}

\def\N{\mathbb{N}}

\def\R{\mathbb{R}}

\def\ra{\rightarrow}
\def\bs{\backslash}

\def\al{\alpha}

\def\ep{\epsilon}
\def\ka{\kappa}

\def\la{\lambda}

\def\om{\omega}

\def\de{\delta}

\def\ga{\gamma}
\def\Ga{\Gamma}

\begin{document}

\title{Regularity of Transition Fronts in Nonlocal Dispersal Evolution Equations}

\author{Wenxian Shen\footnote{Email: wenxish@auburn.edu}\quad and\quad Zhongwei Shen\footnote{Email: zzs0004@auburn.edu, zhongwei@ualberta.ca} \footnote{Current address: Department of Mathematical and Statistical Sciences, University of Alberta, Edmonton, AB T6G 2G1, Canada}\\Department of Mathematics and Statistics\\Auburn University, AL 36849\\USA}

\date{}

\maketitle

\noindent {\bf Abstract.} It is known that solutions of nonlocal dispersal evolution equations do not become smoother in space as time elapses. This lack of space regularity would cause a lot of  difficulties in studying transition fronts in nonlocal equations. In the present paper, we establish some general criteria concerning space regularity of transition fronts in nonlocal dispersal evolution equations with a large class of nonlinearities, which allows the applicability of various techniques for reaction-diffusion equations to nonlocal equations, and hence serves as an initial and fundamental step for further studying various important qualitative properties of transition fronts such as stability, uniqueness and asymptotic speeds. We also prove the existence of continuously differentiable and increasing interface location functions, which give a better characterization of the propagation of transition fronts and are of great technical importance.

\medskip


\medskip


\tableofcontents


\section{Introduction}

Reaction-diffusion equations of the form
\begin{equation}\label{eqn-classical}
u_{t}=u_{xx}+f(t,x,u),\quad (t,x)\in\R\times\R
\end{equation}
are widely used to model diffusive systems in applied sciences.
The nonlinearity $f$ arising from many diffusive systems in biology or physics possesses two zeros representing two states, say $0$ and $1$, that is, $f(t,x,0)=f(t,x,1)=0$ for all $(t,x)\in\R\times\R$.
Since the pioneering works of Fish (see \cite{Fish37}) and Kolmogorov, Petrowsky and Piscunov (see \cite{KPP37}) on the traveling waves of \eqref{eqn-classical} connecting the two constant states, i.e., $u\equiv 0$ and $u\equiv 1$, in the case $f(t,x,u)=u(1-u)$, a vast amount of literature has been carried out to the understanding of front-like solutions connecting $u\equiv 1$ and $u\equiv 0$ in such an equation and its generalized forms. We refer to \cite{ArWe75,ArWe78,FiMc77,FiMc80,Kam76,Uch78,Xi00} and references therein for works in homogeneous media, i.e., $f(t,x,u)=f(u)$. Recently, there is a lot of progress concerning \eqref{eqn-classical} in heterogeneous media. We refer to \cite{BeHa02,DiHaZh14,MeRoSi10,MNRR09,Na14,NoRy09,NRRZ12,TaZhZl14,We02,Zl12} and references therein for works in space heterogeneous media, i.e., $f(t,x,u)=f(x,u)$, and to \cite{AlBaCh99,NaRo12,Sh99-1,Sh99-2,Sh06,Sh11,ShSh14,ShSh14-1} and references therein for works in time heterogeneous media, i.e., $f(t,x,u)=f(t,x)$. There are also some works in space-time heterogeneous media (see e.g. \cite{KoSh14,LiZh1,LiZh2,Na09,Na10, Sh04,Sh11-1,We02}), but it remains widely open.

When using equation \eqref{eqn-classical} to model a diffusive system in applied sciences, it is implicitly assumed that the internal
interaction range of organisms in the system is infinitesimal  and that the internal dispersal can be described by random walk.
 However, in practice,  a diffusive system may exhibit long range internal interaction. Equation \eqref{eqn-classical} is then no long suitable to model such a system. More precisely, the random dispersal operator $\partial_{xx}$ is no long suitable. As a substitute, the nonlocal dispersal operator is introduced (see e.g. \cite{Fi03,GHHMV05} for some background) and we are now concerned with the following integral equation
\begin{equation}\label{main-eqn}
u_{t}=J\ast u-u+f(t,x,u),\quad (t,x)\in\R\times\R,
\end{equation}
where $J$ is a convolution kernel and $[J\ast u](x)=\int_{\R}J(x-y)u(y)dy=\int_{\R}J(y)u(x-y)dy$.
There is also a great amount of research toward  the understanding of front-like solutions of \eqref{main-eqn} connecting $u\equiv 0$ and $u\equiv 1$. See \cite{BaCh06, BaCh02, BaFiReWa97,CaCh04,Ch97,Cov-thesis,CoDu05,CoDu07,Sc80} and references therein for the study in the  homogeneous case $f(t,x,u)\equiv f(u)$.  See \cite{BaCh99, FCh} for the study in the case that $f(t,x,u)\equiv f(t,u)$ is periodic or almost periodic in $t$. In \cite{CDM13,ShZh10,ShZh12-1,ShZh12-2}, the authors investigated \eqref{main-eqn} in space periodic monostable media, i.e., $f(t,x,u)=f(x,u)$ is of monostable type and periodic in $x$, and proved the existence of spreading speeds and periodic traveling waves. In \cite{RaShZh}, the authors studied the existence of spreading speeds and traveling waves of \eqref{main-eqn} in space-time periodic monostable media. Very recently, both Berestycki, Coville and Vo (see \cite{BCV14}), and Lim and Zlato\v{s} (see \cite{LiZl14}) investigated \eqref{main-eqn} in space heterogeneous monostable media. While Berestycki, Coville and Vo studied principal eigenvalue, positive solution and long-time behavior of solutions, Lim and Zlato\v{s} proved the existence of transition fronts in the sense of Berestycki-Hamel (see \cite{BeHa07,BeHa12}). In \cite{ShSh14-2,ShSh14-3}, the authors studied \eqref{main-eqn} in the time heterogeneous media of ignition type, and prove the existence, regularity and stability of transition fronts.

However, comparing to the classical random dispersal case, i.e., \eqref{eqn-classical}, results concerning front propagation for \eqref{main-eqn} are still very limited. One of the difficulties arising in the study of front propagation dynamics of \eqref{main-eqn} is that solutions of \eqref{main-eqn} do not become smoother as time $t$ elapses due to the fact that the semigroup generated by the nonlocal dispersal operator $u\mapsto J\ast u-u$ has no regularizing effect. The objective of the present paper is to investigate the space regularity of transition fronts of \eqref{main-eqn} with various nonlinearities, including the monostable nonlinearity, the ignition nonlinearity, and the bistable nonlinearity. The results to be developed in this paper have important applications in the study of stability, uniqueness, asymptotic, etc. of transition fronts of \eqref{main-eqn}.

To state the main results of this paper, we first introduce two standard hypotheses.

\begin{itemize}
\item[\bf{(H1)}] $J:\R\to\R$ satisfies $J\not\equiv0$, $J\in C^{1}$, $J(x)=J(-x)\geq0$ for $x\in\R$, $\int_{\R}J(x)dx=1$ and
\begin{equation*}
\int_{\R}J(x)e^{\ga x}dx<\infty,\quad \int_{\R}|J'(x)|e^{\ga x}dx<\infty,\quad\forall\ga\in\R.
\end{equation*}
\end{itemize}

\begin{itemize}
\item[\bf{(H2)}] There exist $C^{2}$ functions $f_{B}:[0,1]\to\R$ and  $f_{M}:[0,1]\to\R$ such that
\begin{equation*}
f_{B}(u)\leq f(t,x,u)\leq f_{M}(u),\quad (t,x,u)\in\R\times\R\times[0,1];
\end{equation*}
moreover, the following conditions hold:
\begin{itemize}
\item $f:\R\times\R\times[0,1]\to\R$ is continuous and continuously differentiable in $x$ and $u$, and satisfies
\begin{equation*}
\sup_{(t,x,u)\in\R\times\R\times[0,1]}|f_{x}(t,x,u)|<\infty\quad\text{and}\quad\sup_{(t,x,u)\in\R\times\R\times[0,1]}|f_{u}(t,x,u)|<\infty;
\end{equation*}

\item there exist $\theta_{1}\in(0,1)$ such that $f_{u}(t,x,u)\leq0$ for all $(t,x,u)\in\R\times\R\times[\theta_{1},1]$;

\item $f_{B}$ is of standard bistable type, that is, $f_{B}(0)=f_{B}(\theta)=f_{B}(1)=0$ for some $\theta\in(0,1)$, $f_{B}(u)<0$ for $u\in(0,\theta)$, $f_{B}(u)>0$ for $u\in(\theta,1)$ and $\int_{0}^{1}f_{B}(u)du>0$; moreover, $u_{t}=J\ast u-u+f_{B}(u)$ admits a traveling wave $\phi_{B}(x-c_{B}t)$ with $\phi_{B}(-\infty)=1$, $\phi_{B}(\infty)=0$ and $c_{B}\neq0$;

\item  $f_{M}$ is of standard monostable type, that is, $f_{M}(0)=f_{M}(1)=0$ and $f_{M}(u)>0$ for $u\in(0,1)$.
\end{itemize}
\end{itemize}

We remark that $c_{B}$ must be positive. Observe that nonlinear functions satisfying (H2) include the monostable nonlinearity, the ignition nonlinearity, and the bistable nonlinearity satisfying the conditions below.

\begin{itemize}
\item[\rm(i)] Monostable or Fisher-KPP nonlinearity.

\item[] Standard monostable nonlinearity $f(\cdot)$: $f(0)=f(1)=0$, $f(u)>0$ for $u\in(0,1)$, and $\frac{f(u)}{u}$ is decreasing in $u$.

\item[] General monostable nonlinearity $f(\cdot,\cdot,\cdot)$: $f_{\min}(u)\le f(t,x,u)\le f_{\max}(u)$ for $(t,x)\in\R\times\R$ and $u\in[0,1]$, where $f_{\min}(\cdot)$ and $f_{\max}(u)$ are two standard monostable nonlinearities, and $\frac{f(t,x,u)}{u}$ decreasing in $u$ for any $(t,x)\in\R\times\R$.

\item[{\rm (ii)}] Ignition nonlinearity.

\item[] Standard ignition nonlinearity $f(\cdot)$: $f(u)=0$ for $u\in [0,\theta]\cup\{1\}$,
$f(u)>0$ for $u\in(\theta,1)$, and $f_u(1)<0$, where $\theta\in(0,1)$ is referred to as the {\it ignition temperature}.

\item[] General ignition nonlinearity $f(\cdot,\cdot,\cdot)$: $f_{\min}(u)\le f(t,x,u)\le f_{\max}(u)$ for $(t,x)\in\R\times\R$ and $u\in[0,1]$,
where $f_{\min}(\cdot)$ and $f_{\max}(\cdot)$ are two standard ignition nonlinearities.

\item[{\rm (iii)}] Bistable nonlinearity.

\item[] Standard bistable nonlinearity $f(\cdot)$: $f(0)=f(\theta)=f(1)=0$, $f(u)<0$ for $u\in(0,\theta)$, $f(u)>0$ for $u\in(\theta,1)$, and $f_u(0)<0$, $f_u(1)<0$, $f_u(\theta)>0$,
where $\theta\in(0,1)$.

\item[] General bistable nonlinearity $f(\cdot,\cdot,\cdot)$: $f_{\min}(u)\le f(t,x,u)\le f_{\max}(u)$ for $(t,x)\in\R\times\R$ and $u\in [0,1]$, where $f_{\min}(\cdot)$ and $f_{\max}(\cdot)$ are two standard bistable nonlinearities with $\int_{0}^{1}f_{\min}(u)du>0$ and $u_{t}=J\ast u-u+f_{\min}(u)$ having traveling waves with nonzero speed.
 \end{itemize}

We remark that (H2) can also be applied to a general bistable nonlinearity $f(t,x,u)$ with $\int_0^1 f_{\max}(u)du<0$ and $u_{t}=J\ast u-u+f_{\max}(u)$ having traveling waves with nonzero speed. In fact, let $v(t,x)=1-u(t,x)$. Then, $v(t,x)$ satisfies
$$
v_t=J\ast v-v+\tilde f(t,x,v),\quad (t,x)\in\R\times\R,
$$
where $\tilde f(t,x,v)=-f(t,x,1-v)$. Hence
$$
\tilde f_{\min}(v)\le \tilde f(t,x,v)\le \tilde f_{\max}(v),\quad (t,x,v)\in\R\times\R\times[0,1]
$$
where $\tilde f_{\min}(v)=-f_{\max}(1-v)$ and $\tilde f_{\max}(v)=-f_{\min}(1-v)$.
Clearly, $\tilde f_{\min}(\cdot)$ and $\tilde f_{\max}(\cdot)$ are two standard bistable nonlinearities and
$\int_0^1 \tilde f_{\min}(v)dv>0$ and $u_{t}=J\ast u-u+\tilde{f}_{\min}(u)$ admits traveling waves with nonzero speed.

We also remark that, besides monostable, ignition, and bistable nonlinearities, nonlinear functions satisfying (H2) include those having more than one zeros between $0$ and $1$.

Next, we recall the definition of transition fronts of \eqref{main-eqn} connecting $u\equiv 0$ and $u\equiv 1$.

\begin{defn}\label{transition-front-defn}
Suppose that $f(t,x,0)=f(t,x,1)=0$ for all $(t,x)\in\R\times\R$. A global-in-time solution $u(t,x)$ of \eqref{main-eqn} is called a (right-moving) {\rm transition front} (connecting $0$ and $1$) in the sense of Berestycki-Hamel (see \cite{BeHa07,BeHa12}, also see \cite{Sh99-1,Sh99-2}) if $u(t,x)\in(0,1)$ for all $(t,x)\in\R\times\R$ and there exists a function $X:\R\to\R$, called {\rm interface location function}, such that
\begin{equation*}
\lim_{x\to-\infty}u(t,x+X(t))=1\,\,\text{and}\,\,\lim_{x\to\infty}u(t,x+X(t))=0\,\,\text{uniformly in}\,\,t\in\R.
\end{equation*}
\end{defn}

The notion of a transition front is a proper generalization of a traveling wave in homogeneous media and a periodic (or pulsating) traveling wave in periodic media. The interface location function $X(t)$ tells the position of the transition front $u(t,x)$ as time $t$ elapses, while the uniform-in-$t$ limits (the essential property in the definition) shows the \textit{bounded interface width}, that is,
\begin{equation*}
\forall\,\,0<\ep_{1}\leq\ep_{2}<1,\quad\sup_{t\in\R}{\rm diam}\{x\in\R|\ep_{1}\leq u(t,x)\leq\ep_{2}\}<\infty.
\end{equation*}
Notice, if $\xi(t)$ is a bounded function, then $X(t)+\xi(t)$ is also an interface location function. Thus, interface location function is not unique. But, it is easy to check that if $Y(t)$ is another interface location function, then $X(t)-Y(t)$ is a bounded function. Hence, interface location functions are unique up to addition by bounded functions.

We see that neither the definition nor the equation \eqref{main-eqn} guarantees any space regularity of transition fronts. In fact, there is even no guarantee that a transition front $u(t,x)$ of \eqref{main-eqn} is continuous in $x$ (we refer to \cite{BaFiReWa97} for the existence of discontinuous traveling waves in the bistable case). This lack of space regularity indeed causes a lot of troubles in studying transition fronts because $\rm(i)$ space regularity of approximating solutions is required to ensure the convergence to transition fronts; $\rm(ii)$ space regularity of transition fronts lays the foundation for applying various techniques for reaction-diffusion equations to nonlocal equations, and hence, for further studying various qualitative properties such as stability and uniqueness. Hence, it is very important to study the space regularity of special solutions.

In the present paper, we intend to establish some general criteria concerning the space regularity of transition fronts of \eqref{main-eqn}. More precisely, we want to know whether a transition front $u(t,x)$ of \eqref{main-eqn} is continuously differentiable in $x$. To state our first result, we further introduce the following hypothesis.

\begin{itemize}
\item[\bf{(H3)}] There exist $\theta_{0}\in(0,\theta_{1})$ and $\ka_{0}>0$ such that
\begin{equation*}
f_{u}(t,x,u)\leq1-\ka_{0},\quad (t,x,u)\in\R\times\R\times[0,\theta_{0}].
\end{equation*}
\end{itemize}

We prove

\begin{thm}\label{thm-regularity-of-tf}
Suppose (H1)-(H3). Let $u(t,x)$ be an arbitrary transition front of \eqref{main-eqn}. Then, $u(t,x)$ is regular in space in the following sense:  for any $t\in\R$, $u(t,x)$ is continuously differentiable in $x$ and satisfies $\sup_{(t,x)\in\R\times\R}|u_{x}(t,x)|<\infty$.
\end{thm}

We remark that Theorem \ref{thm-regularity-of-tf} can be proven essentially due to the assumptions $\int_{0}^{1}f_{B}(u)du>0$ and the existence of traveling waves of 
\begin{equation}\label{bistable-tw-eqn-000000}
u_{t}=J\ast u-u+f_{B}(u)
\end{equation}
with nonzero speed in (H2), which corresponds to the unbalanced case in the bistable case. If we drop these assumptions, then, in the bistable case, there is no hope that Theorem \ref{thm-regularity-of-tf} is true without additional assumptions on $f(t,x,u)$, since discontinuous traveling waves of \eqref{bistable-tw-eqn-000000} with zero speed were constructed in \cite{BaFiReWa97}, where necessary conditions for the existence of discontinuous traveling waves are also given. Thus, our results are kind of sharp and are compatible with the results for traveling waves in the monostable case, the ignition case and the unbalanced bistable case. It is worthwhile to point out that in the case \eqref{bistable-tw-eqn-000000} admits discontinuous traveling waves, it may also admits non-monotone waves (see \cite[Theorem 5.4]{Ch97}).

Although the assumptions on $f_{B}$ are automatically true in the monostable case and the ignition case, it does cause some restrictions in the bistable case, and hence, there remains an interesting problem, that is, whether transition fronts in the bistable case are regular in space when \eqref{bistable-tw-eqn-000000} admits only smooth stationary waves, i.e., smooth traveling waves with speed zero. 

The proof of Theorem \ref{thm-regularity-of-tf} is based on the rightward propagation estimate of transition fronts and the analysis of the  growth and the  decay of $\frac{u(t,x+\eta)-u(t,x)}{\eta}$. The rightward propagation estimate reads as
\begin{equation}\label{right-propagation-intro}
X(t)-X(t_{0})\geq c_{1}(t-t_{0}-T_{1}),\quad t\geq t_{0}
\end{equation}
for some $c_{1}>0$ and $T_{1}>0$, which is established in Theorem \ref{lem-propagation-estimate} (to show \eqref{right-propagation-intro},  we need $\int_{0}^{1}f_{B}(u)du>0$ and the existence of traveling waves of \eqref{bistable-tw-eqn-000000} with nonzero speed). To control the behavior of $\frac{u(t,x+\eta)-u(t,x)}{\eta}$ when $x$ is close to $\infty$, we need (H3). A key ingredient in proving Theorem \ref{thm-regularity-of-tf}, i.e., controlling the term $\frac{u(t,x+\eta)-u(t,x)}{\eta}$, is the observation that for fixed $x$, the term $\frac{u(t,x+\eta)-u(t,x)}{\eta}$ can only grow for a period of time that is independent of $x$. Moreover, since we directly study transition fronts, which may not come from approximating solutions, we are lack of a priori information, which immediately causes the possible blow-up behavior of $\frac{u(t_{0},x+\eta)-u(t_{0},x)}{\eta}$ as $\eta\to0$ at the initial time $t_{0}$. To overcome this technical difficulty, we utilize the fact that transition fronts are global-in-time, which means we can take $t_{0}$ to approach $-\infty$ along subsequences.

Clearly, (H3) rules out many monostable nonlinearities, and it does not cover all Fisher-KPP nonlinearities. Our next result is trying to cover the monostable nonlinearities at the cost of putting some restrictions on transition fronts. We prove

\begin{thm}\label{thm-regularity-of-tf-ii}
Suppose (H1) and (H2). Suppose, in addition, that
\begin{equation*}
\sup_{(t,x,u)\in\R\times\R\times[0,1]}|f_{xu}(t,x,u)|<\infty\quad\text{and}\quad\sup_{(t,x,u)\in\R\times\R\times[0,1]}|f_{uu}(t,x,u)|<\infty.
\end{equation*}
Let $u(t,x)$ be an arbitrary transition front of \eqref{main-eqn} satisfying
\begin{equation}\label{harnack-type-intro}
u(t,x)\leq Ce^{r|x-y|}u(t,y),\quad (t,x,y)\in\R\times\R\times\R
\end{equation}
for some $C>0$ and $r>0$. Then, $u(t,x)$ is regular in space in the following sense:  for any $t\in\R$, $u(t,x)$ is continuously differentiable in $x$ and satisfies $\sup_{(t,x)\in\R\times\R}\frac{|u_{x}(t,x)|}{u(t,x)}<\infty$.
\end{thm}

The key assumption in Theorem \ref{thm-regularity-of-tf-ii} is the Harnack-type inequality \eqref{harnack-type-intro}, which is the case for some transition fronts in the Fisher-KPP case (see \cite{LiZl14,ShSh14-kpp}). The importance of \eqref{harnack-type-intro} lies in the fact that it allows the comparison of $J\ast u$ and $J'\ast u$ with $u$. More precisely, by \eqref{harnack-type-intro}, we find $\frac{1}{C}\int_{\R}J(x)e^{-r|x|}dx\leq\frac{J\ast u}{u}\leq C\int_{\R}J(x)e^{r|x|}dx$ and similarly for $J'\ast u$, which plays crucial roles in controlling $\frac{u(t,x+\eta)-u(t,x)}{\eta u(t,x)}$.

The next result gives space regularity of transition fronts under the exact decay assumption.

\begin{thm}\label{thm-regularity-of-tf-iii}
Suppose (H1) and (H2). Suppose, in addition, that
\begin{equation*}
\sup_{(t,x,u)\in\R\times\R\times[0,1]}|f_{xu}(t,x,u)|<\infty\quad\text{and}\quad\sup_{(t,x,u)\in\R\times\R\times[0,1]}|f_{uu}(t,x,u)|<\infty.
\end{equation*}
Let $u(t,x)$ be an arbitrary transition front of \eqref{main-eqn} with interface location function $X(t)$. There exists $r_{0}>0$ such that if 
\begin{equation}\label{condition-exact-decay-intro}
\lim_{x\to\infty}\frac{u(t,x+X(t))}{e^{-rx}}=1\quad\text{uniformly in}\quad t\in\R
\end{equation}
for some  $r\in(0,r_{0}]$,
then, $u(t,x)$ is regular in space in the following sense:  for any $t\in\R$, $u(t,x)$ is continuously differentiable in $x$ and satisfies $\sup_{(t,x)\in\R\times\R}\frac{|u_{x}(t,x)|}{u(t,x)}<\infty$.
\end{thm}

We remark that the exact decay assumption \eqref{condition-exact-decay-intro} is the case for some transition fronts in the Fisher-KPP case (see \cite{LiZl14,ShSh14-kpp}). The importance of \eqref{condition-exact-decay-intro} lies in the fact that it allows the comparison of $J\ast u$ and $J'\ast u$ with $u$ near $x=\infty$. Note that if the limit in \eqref{condition-exact-decay-intro} is some other positive number instead of $1$, then we only need to shift the exponential function on the bottom correspondingly. The number $r_{0}$ corresponds to the possible decay rate of Fisher-KPP transition fronts, and thus, it would be interesting to determine the optimal $r_{0}$.

The study of space regularity of transition fronts of \eqref{main-eqn} was initiated in \cite{ShSh14-3}, where the space regularity of well-constructed transition fronts in time heterogeneous media of ignition type is obtained. Those well-constructed transition fronts are uniformly Lipschitz continuous in space and their interface location functions can be chosen to be continuously differentiable with uniformly positive first order derivatives. These properties, which may not be true for an arbitrary transition front, play important roles in the study of space regularity of well-constructed transition fronts in \cite{ShSh14-3}. Our results, Theorem \ref{thm-regularity-of-tf}, Theorem \ref{thm-regularity-of-tf-ii} and Theorem \ref{thm-regularity-of-tf-iii}, generalize the work in \cite{ShSh14-3} to arbitrary transition fronts (with some additional assumptions in Theorem \ref{thm-regularity-of-tf-ii} and Theorem \ref{thm-regularity-of-tf-iii}) of \eqref{main-eqn} with a large class of nonlinearities. As already shown in \cite{ShSh14-3}, space regularity of transition fronts is of great significance in the study of stability, which together with uniqueness and asymptotic speeds, will be studied elsewhere.

Finally, we study the existence of continuously differentiable and increasing interface location functions. As mentioned before, if $X(t)$ is an interface location function of a transition front $u(t,x)$ of \eqref{main-eqn}, then for any bounded function $\xi(t)$, $X(t)+\xi(t)$ is also an interface location function of $u(t,x)$. Hence, interface location functions of a transition front are not unique and may not be continuous. But, near each interface location function, we are able to find a continuously differentiable and increasing function, as the new interface location function. This is given by the following

\begin{thm}\label{thm-modified-interface-location}
Suppose (H1) and (H2). Let $u(t,x)$ be an arbitrary transition front of \eqref{main-eqn} with interface location function $X(t)$ satisfying
\begin{equation}\label{upper-avg-speed}
X(t)-X(t_{0})\leq c_{2}(t-t_{0}+T_{2}),\quad t\geq t_{0}
\end{equation}
for some $c_{2}>0$ and $T_{2}>0$. Then, there are constants $0<\tilde{c}_{\min}\leq\tilde{c}_{\max}<\infty$ and a continuously differentiable function $\tilde{X}(t)$ satisfying
\begin{equation*}
\tilde{c}_{\min}\leq\dot{\tilde{X}}(t)\leq \tilde{c}_{\max},\quad t\in\R
\end{equation*}
such that
\begin{equation*}
\sup_{t\in\R}|\tilde{X}(t)-X(t)|<\infty.
\end{equation*}
In particular, $\tilde{X}(t)$ is also an interface location function of $u(t,x)$.
\end{thm}

We see that (H1) and (H2) do not ensure the space regularity of  transition fronts. It will be clear later that the proof of Theorem \ref{thm-modified-interface-location} does not need the space regularity of transition fronts.

We refer to $\tilde{X}(t)$ in Theorem \ref{thm-modified-interface-location} as the modified interface location function, which gives a better characterization of the propagation of transition fronts and has been verified to be of great technical importance in studying the stability of transition fronts in time heterogeneous media (see e.g. \cite{ShSh14-1,ShSh14-2,ShSh14-3}). Moreover, for reaction-diffusion equations in space heterogeneous media, the rightmost interface location at some constant value continuously moves to the right (see e.g. \cite{MeRoSi10,MNRR09,NoRy09,Zl13}). But for nonlocal equations in space heterogeneous media, the rightmost interface location at some constant value jumps in general due to the nonlocality, which makes the modified interface location function more important.

The condition \eqref{upper-avg-speed} is a technical assumption saying a transition front moves to the right at most at linear speed in the average sense, which together with \eqref{right-propagation-intro} allow us to find the modified interface location function. Arguing as in the proof of Theorem \ref{lem-propagation-estimate}, we readily check that the condition \eqref{upper-avg-speed} is always true in the bistable case. This condition can also be verified in other cases including the monostable case and the ignition case as in the following two corollaries.

\begin{cor}\label{cor-modified-interface}
Suppose (H1) and (H2). Let $u(t,x)$ be an arbitrary transition front of \eqref{main-eqn}. If there exists $\tilde{\theta}\in(0,\theta)$ such that
\begin{equation*}
f(t,x,u)\leq0,\quad(t,x,u)\in\R\times\R\times[0,\tilde{\theta}],
\end{equation*}
then \eqref{upper-avg-speed} is true. In particular, the conclusions of Theorem \ref{thm-modified-interface-location} hold.
\end{cor}

Clearly, Corollary \ref{cor-modified-interface} covers, in particular, all bistable and ignition nonlinearities, but rules out all monostable nonlinearities, which are covered by the next result.

\begin{cor}\label{cor-modified-interface-ii}
Suppose (H1) and (H2). Suppose, in addition, that $J$ is compactly supported. Let $u(t,x)$ be an arbitrary transition front of \eqref{main-eqn} with interface location function $X(t)$. If there exist $r>0$ and $h>0$ such that
\begin{equation*}
u(t,x+X(t))\leq e^{-r(x-h)},\quad (t,x)\in\R\times\R,
\end{equation*}
then \eqref{upper-avg-speed} is true. In particular, the conclusions of Theorem \ref{thm-modified-interface-location} hold.
\end{cor}

We remark that in the monostable case, uniform exponential decay as $x\to\infty$ may be necessary for transition fronts to travel at linear speeds, since slower decay near $x=\infty$ may cause super-linear propagation (see \cite{HaRo10}). It is worthwhile to point out that in the bistable case, a discontinuous traveling wave may not converge to zero exponentially as $x\to\infty$ (see \cite[Theorem 5.1]{BaChm99}). Here, we need $J$ to be compactly supported, since we will use results obtained in \cite{CaCh04} and \cite{ShZh12-2} in the proof, which were proven when $J$ is compactly supported. It should be pointed out that the arguments and results in
\cite{ShZh12-2} can be extended to dispersal kernels $J$ which are not compactly supported, but satisfy (H1) (see \cite{AZh}).

As a direct application of the results, in particular, Theorem \ref{thm-regularity-of-tf} and Corollary \ref{cor-modified-interface}, obtained in the present paper, we study in \cite{ShSh14-bistable} the equation \eqref{main-eqn} in time heterogeneous media of general bistable type, that is, $f(t,x,u)=f(t,u)$ satisfies
\begin{equation*}
f_{\min}(u)\leq f(t,u)\leq f_{\max}(u),\quad (t,u)\in\R\times[0,1],
\end{equation*}
where the $C^{2}$ functions $f_{\min}$ and $f_{\max}$ are standard bistable nonlinearities on $[0,1]$ with the unbalanced condition $\int_{0}^{1}f_{\min}(u)du>0$. Provided that there is a space nonincreasing transition front, we show by means of Theorem \ref{thm-regularity-of-tf} and Corollary \ref{cor-modified-interface} that (i) all transition fronts are asymptotically stable and enjoy decaying estimates; (ii) transition fronts are unique up to space shifts; (iii) all transition fronts become periodic traveling waves in the periodic media and have asymptotic speeds in the uniquely ergodic media. The assumption on the existence of a space nonincreasing transition front can be verified if, for example, $f(t,u)$ is of standard bistable type in $u$ for each $t$ and their middle zeros are the same.

The rest of the paper is organized as follows. In Section \ref{sec-prop-estimate}, we establish the rightward propagation estimate of transition fronts. We will see, in particular, that any transition fronts moves front left infinity to right infinity as time goes from $-\infty$ to $\infty$. In Section \ref{subsec-unconditional-regularity}, we prove Theorem \ref{thm-regularity-of-tf}. In Section \ref{sec-proof-2}, we prove Theorem \ref{thm-regularity-of-tf-ii}. In Section \ref{sec-proof-iafjiaf}, we prove Theorem \ref{thm-regularity-of-tf-iii}. In Section \ref{sec-proof-thm-cor}, we prove Theorem \ref{thm-modified-interface-location}, Corollary \ref{cor-modified-interface} and Corollary \ref{cor-modified-interface-ii}.  We end up the present paper with an appendix, Appendix \ref{app-ig-tw}, on ignition traveling waves.


\section{Rightward propagation estimates}\label{sec-prop-estimate}

In this section, we study the rightward propagation estimates of transition fronts. Throughout this section, we assume (H1) and (H2).

In what follows in this section, $u(t,x)$ will be an arbitrary transition front of \eqref{main-eqn} with interface location function $X(t)$. For $\la\in(0,1)$, we define $X_{\la}^{-}(t)$ and $X_{\la}^{+}(t)$ by setting
\begin{equation}\label{defn-interface-locations}
\begin{split}
X_{\la}^{-}(t)&=\sup\{x\in\R|u(t,y)>\la,\,\,\forall y\leq x\},\\
 X_{\la}^{+}(t)&=\inf\{x\in\R|u(t,y)<\la,\,\,\forall y\geq x\}.
\end{split}
\end{equation}
Note that if $u(t,x)$ is continuous in $x$, then $X_{\la}^{-}(t)$ and $X_{\la}^{+}(t)$ are nothing but the leftmost and rightmost interface locations at $\la$. Trivially, $X_{\la}^{-}(t)\leq X_{\la}^{+}(t)$ and $X_{\la}^{\pm}(t)$ are decreasing in $\la$. Due to the possible discontinuity of $u(t,x)$ in $x$, it may happen that $u(t,X_{\la}^{-}(t))<\la$ or $u(t,X_{\la}^{+}(t))>\la$.

From the definition of transition fronts, we have the following simple lemma.

\begin{lem}\label{lem-bounded-interface-width}
The following statements hold:
\begin{itemize}
\item[\rm(i)] for any $0<\la_{1}\leq\la_{2}<1$, there holds $\sup_{t\in\R}[X^{+}_{\la_{1}}(t)-X^{-}_{\la_{2}}(t)]<\infty$;

\item[\rm(ii)] for any $\la\in(0,1)$, there hold $\sup_{t\in\R}|X(t)-X^{\pm}_{\la}(t)|<\infty$.
\end{itemize}
\end{lem}
\begin{proof}
$\rm(i)$ By the uniform-in-$t$ limits $\lim_{x\to-\infty}u(t,x+X(t))=1$ and $\lim_{x\to\infty}u(t,x+X(t))=0$, there exist $x_{1}$ and $x_{2}$ such that $u(t,x+X(t))>\la_{2}$ for all $x\leq x_{2}$ and $t\in\R$, and $u(t,x+X(t))<\la_{1}$ for all $x\geq x_{1}$ and $t\in\R$. It then follows from the definition of $X_{\la_{2}}^{-}(t)$ and $X_{\la_{1}}^{+}(t)$ that $x_{2}+X(t)\leq X_{\la_{2}}^{-}(t)$ and $x_{1}+X(t)\geq X_{\la_{1}}^{+}(t)$ for all $t\in\R$. The result then follows.

$\rm(ii)$ Let $\la_{1}=\la=\la_{2}$ in the proof of $\rm(i)$, we have  $x_{2}+X(t)\leq X_{\la}^{-}(t)$ and $x_{1}+X(t)\geq X_{\la}^{+}(t)$ for all $t\in\R$. In particular,
\begin{equation*}
x_{2}+X(t)\leq X_{\la}^{-}(t)\leq X_{\la}^{+}(t)\leq x_{1}+X(t),\quad t\in\R.
\end{equation*}
This completes the proof.
\end{proof}

The next result gives the rightward propagation estimate of $u(t,x)$ in terms of $X(t)$.

\begin{thm}\label{lem-propagation-estimate}
There exist $c_{1}>0$ and $T_{1}>0$ such that
\begin{equation*}
X(t)-X(t_{0})\geq c_{1}(t-t_{0}-T_{1}),\quad t\geq t_{0}.
\end{equation*}
\end{thm}
\begin{proof}
Fix some $\la\in(\theta,1)$. We  write $X^{-}(t)=X^{-}_{\la}(t)$. Since $\sup_{t\in\R}|X(t)-X^{-}(t)|<\infty$ by Lemma \ref{lem-bounded-interface-width}, it suffices to show
\begin{equation}\label{propagation-estimate-1}
X^{-}(t)-X^{-}(t_{0})\geq c(t-t_{0}-T),\quad t\geq t_{0}
\end{equation}
for some $c>0$ and $T>0$.

Recall $f_{B}$ is as in (H2). Let $(c_{B},\phi_{B})$ with $c_{B}>0$ be the unique solution of
\begin{equation*}
\begin{cases}
J\ast\phi-\phi+c\phi_{x}+f_{B}(\phi)=0,\\
\phi_{x}<0,\,\, \phi(0)=\theta,\,\, \phi(-\infty)=1\,\,\text{and}\quad\phi(\infty)=0
\end{cases}
\end{equation*}
(see \cite{BaFiReWa97} for the existence and uniqueness of $(c_B,\phi_B)$). That is, $c_{B}$ is the unique speed and $\phi_{B}$ is the normalized profile of traveling waves of
\begin{equation}\label{eqn-bistable-homo-B}
u_{t}=J\ast u-u+f_{B}(u).
\end{equation}

Let $u_{0}:\R\to[0,1]$ be a uniformly continuous and nonincreasing function satisfying
\begin{equation}\label{some-special-initial-data-132184}
u_{0}(x)=\begin{cases}
\la,\quad& x\leq x_{0},\\
0,\quad & x\geq0,
\end{cases}
\end{equation}
where $x_{0}<0$ is fixed. By the definition of $X^{-}(t)$, we see that for any $t_{0}\in\R$, there holds $u(t_{0},x+X^{-}(t_{0}))\geq u_{0}(x)$ for all $x\in\R$, and then, by $f(t,x,u)\geq f_{B}(u)$ and the comparison principle, we find
\begin{equation}\label{result-after-cp-42492}
u(t,x+X^{-}(t_{0}))\geq u_{B}(t-t_{0},x;u_{0}),\quad x\in\R,\,\,t\geq t_{0},
\end{equation}
where $u_{B}(t,x;u_{0})$ is the unique solution to \eqref{eqn-bistable-homo-B} with $u_{B}(0,\cdot;u_{0})=u_{0}$. By the choice of $u_{0}$ and the stability of bistable traveling waves (see e.g. \cite{BaFiReWa97}), there are constants $x_{B}=x_{B}(\la)\in\R$, $q_{B}=q_{B}(\la)>0$ and $\om_{B}>0$ such that
\begin{equation*}
u_{B}(t-t_{0},x;u_{0})\geq\phi_{B}(x-x_{B}-c_{B}(t-t_{0}))-q_{B}e^{-\om_{B}(t-t_{0})}, \quad x\in\R,\,\,t\geq t_{0}.
\end{equation*}
Hence,
\begin{equation*}
u(t,x+X^{-}(t_{0}))\geq \phi_{B}(x-x_{B}-c_{B}(t-t_{0}))-q_{B}e^{-\om_{B}(t-t_{0})}, \quad x\in\R,\,\,t\geq t_{0}.
\end{equation*}
Let $T_{0}=T_{0}(\la)>0$ be such that $q_{B}e^{-\om_{B}T_{0}}=\frac{1-\la}{2}$ (making $q_{B}$ larger so that $q_{B}>\frac{1-\la}{2}$ if necessary) and denote by $\xi_{B}(\frac{1+\la}{2})$ the unique point such that $\phi_{B}(\xi_{B}(\frac{1+\la}{2}))=\frac{1+\la}{2}$. Setting $x_{*}=x_{B}+c_{B}(t-t_{0})+\xi_{B}(\frac{1+\la}{2})$, the monotonicity of $\phi_{B}$ implies that for all $t\geq t_{0}+T_{0}$ and $x\leq x^{*}-1$
\begin{equation*}
\begin{split}
u(t,x+X^{-}(t_{0}))&\geq\phi_{B}(x_{*}-1-x_{B}-c_{B}(t-t_{0}))-q_{B}e^{-\om_{B}T_{0}}\\
&>\phi_{B}(x_{*}-x_{B}-c_{B}(t-t_{0}))-q_{B}e^{-\om_{B}T_{0}}\\
&=\phi_{B}(\xi_{B}(\frac{1+\la}{2}))-q_{B}e^{-\om_{B}T_{0}}=\la.
\end{split}
\end{equation*}
This says that $x_{*}-1+X^{-}(t_{0})\leq X^{-}(t)$ for all $t\geq t_{0}+T_{0}$, that is,
\begin{equation}\label{estimate-long-time}
X^{-}(t)-X^{-}(t_{0})\geq x_{B}-1+c_{B}(t-t_{0})+\xi_{B}(\frac{1+\la}{2}),\quad t\geq t_{0}+T_{0}.
\end{equation}

We now estimate $X^{-}(t)-X^{-}(t_{0})$ for $t\in[t_{0},t_{0}+T_{0}]$. We claim that there exists $z=z(T_{0})<0$ such that
\begin{equation}\label{estimate-finite-time}
X^{-}(t)-X^{-}(t_{0})\geq z,\quad t\in[t_{0},t_{0}+T_{0}].
\end{equation}
Let $u_{B}(t,x;u_{0})$ and $u_{B}(t;\la):=u_{B}(t,x;\la)$ be solutions of \eqref{eqn-bistable-homo-B} with $u_{B}(0,x;u_{0})=u_{0}(x)$ and $u_{B}(0;\la)=u_{B}(0,x;\la)\equiv\la$, respectively. By the comparison principle, we have $u_{B}(t,x;u_{0})<u_{B}(t;\la)$ for all $x\in\R$ and $t>0$, and $u_{B}(t,x;u_{0})$ is strictly decreasing in $x$ for $t>0$.

We see that for any $t>0$, $u_{B}(t,-\infty;u_{0})=u_{B}(t;\la)$. This is because that $\frac{d}{dt}u_{B}(t,-\infty;u_{0})=f_{B}(u_{B}(t,-\infty;u_{0}))$ for $t>0$ and $u_{B}(0,-\infty;u_{0})=\la$. Since $\la\in(\theta,1)$, as a solution of the ODE $u_{t}=f_{B}(u)$, $u_{B}(t;\la)$ is strictly increasing in $t$, which implies that $u_{B}(t,-\infty;u_{0})=u_{B}(t;\la)>\la$ for $t>0$. As a result, for any $t>0$ there exists a unique $\xi_{B}(t)\in\R$ such that $u_{B}(t,\xi_{B}(t);u_{0})=\la$. Moreover, $\xi_{B}(t)$ is continuous in $t$.

Setting $x_{**}=\xi_{B}(t-t_{0})$, we find $u(t,x+X^{-}(t_{0}))>\la$ for all $x\leq x_{**}$ by the monotonicity of $u_{B}(t,x;u_{0})$ in $x$, which  together with \eqref{result-after-cp-42492}
implies that
\begin{equation*}
X^{-}(t)\geq x_{**}+X^{-}(t_{0})=\xi_{B}(t-t_{0})+X^{-}(t_{0}),\quad t>t_{0}.
\end{equation*}
Thus, \eqref{estimate-finite-time} follows if $\inf_{t\in(t_{0},t_{0}+T_{0}]}\xi_{B}(t-t_{0})>-\infty$, that is,
\begin{equation}\label{not-to-negative-finity}
\inf_{t\in(0,T_{0}]}\xi_{B}(t)>-\infty.
\end{equation}

We now show \eqref{not-to-negative-finity}. Since $u_{0}(x)=\la$ for $x\leq x_{0}$, continuity with respect to the initial data
 (in sup norm) implies that for any $\ep>0$ there exists $\de>0$ such that
\begin{equation*}
u_{B}(t;\la)-\la\leq\ep\quad\text{and}\quad\sup_{x\leq x_{0}}[u_{B}(t;\la)-u_{B}(t,x;u_{0})]=u_{B}(t;\la)-u_{B}(t,x_{0};u_{0})\leq\ep
\end{equation*}
for all $t\in[0,\de]$, where the equality is due to monotonicity. By (H1), $J$ concentrates near $0$ and decays very fast as $x\to\pm\infty$. Thus, we can choose $x_{1}=x_{1}(\ep)<<x_{0}$ such that
\begin{equation*}
\int_{-\infty}^{x_{0}}J(x-y)dy\geq1-\ep,\quad x\le x_1.
\end{equation*}
Now, for any $x\leq x_{1}$ and $t\in(0,\de]$, we have
\begin{equation*}
\begin{split}
\frac{d}{dt}u_{B}(t,x;u_{0})&=\int_{\R}J(x-y)u_{B}(t,y;u_{0})dy-u_{B}(t,x;u_{0})+f_{B}(u_{B}(t,x;u_{0}))\\
&\geq\int_{-\infty}^{x_{0}}J(x-y)u_{B}(t,y;u_{0})dy-u_{B}(t,x;u_{0})+f_{B}(u_{B}(t,x;u_{0}))\\
&\geq(1-\ep)\inf_{x\leq x_{0}}u_{B}(t,x;u_{0})-u_{B}(t;\la)+f_{B}(u_{B}(t,x;u_{0}))\\
&=-(1-\ep)\sup_{x\leq x_{0}}[u_{B}(t;\la)-u_{B}(t,x;u_{0})]-\ep u_{B}(t;\la)+f_{B}(u_{B}(t,x;u_{0}))\\
&\geq-\ep(1-\ep)-\ep(\la+\ep)+f_{B}(u_{B}(t,x;u_{0}))>0
\end{split}
\end{equation*}
if we choose $\ep>0$ sufficiently small, since then $f_{B}(u_{B}(t,x;u_{0}))$ is close to $f_{B}(\la)$, which is positive. This simply means that $u_{B}(t,x;u_{0})>\la$ for all $x\leq x_{1}$ and $t\in(0,\de]$, which implies that $\xi_{B}(t)>x_{1}$ for $t\in(0,\de]$. The continuity of $\xi_{B}$ then leads to \eqref{not-to-negative-finity}. This proves \eqref{estimate-finite-time}. \eqref{propagation-estimate-1} then follows from \eqref{estimate-long-time} and \eqref{estimate-finite-time}. This completes the proof.
\end{proof}

As a simple consequence of Theorem \ref{lem-propagation-estimate}, we have

\begin{cor}\label{cor-propagation-estimate}
There holds $X(t)\to\pm\infty$ as $t\to\pm\infty$. In particular,  $u(t,x)\to1$ as $t\to\infty$ and $u(t,x)\to0$ as $t\to-\infty$ locally uniformly in $x$.
\end{cor}
\begin{proof}
We have from Lemma \ref{lem-propagation-estimate} that
\begin{equation*}
X(t)-X(t_{0})\geq c_{1}(t-t_{0}-T_{1}) ,\quad t\geq t_{0}.
\end{equation*}
Setting $t\to\infty$ in the above estimate, we find $X(t)\to\infty$ as $t\to\infty$. Setting $t_{0}\to-\infty$, we find $X(t_{0})\to-\infty$ as $t_{0}\to-\infty$.
\end{proof}

This corollary shows that any transition front travels from the left infinity to the right infinity. Thus, steady-state-like transition fronts, blocking the propagations of solutions, do not exist.

\section{Proof of Theorem \ref{thm-regularity-of-tf}}\label{subsec-unconditional-regularity}

We prove Theorem \ref{thm-regularity-of-tf} in this section. Throughout this section, we assume that (H1)-(H3) hold and $u(t,x)$ is an arbitrary transition front of \eqref{main-eqn} with interface location function $X(t)$.

To prove Theorem \ref{thm-regularity-of-tf}, we first do some preparations and prove several lemmas. Fix some $0<\de_{0}\ll1$. For $(t,x)\in\R\times\R$ and $\eta\in\R$ with $0<|\eta|\leq\de_{0}$, we set
\begin{equation*}
v^{\eta}(t,x)=\frac{u(t,x+\eta)-u(t,x)}{\eta}.
\end{equation*}
It is easy to see that $v^{\eta}(t,x)$ satisfies
\begin{equation}\label{lipschitz-eqn-987654}
v^{\eta}_{t}(t,x)=b^{\eta}(t,x)-v^{\eta}(t,x)+a^{\eta}(t,x)v^{\eta}(t,x)+\tilde{a}^{\eta}(t,x),
\end{equation}
where
\begin{equation*}
\begin{split}
a^{\eta}(t,x)&=\frac{f(t,x,u(t,x+\eta))-f(t,x,u(t,x))}{u(t,x+\eta)-u(t,x)},\\
\tilde{a}^{\eta}(t,x)&=\frac{f(t,x+\eta,u(t,x+\eta))-f(t,x,u(t,x+\eta))}{\eta},\\
b^{\eta}(t,x)&=\int_{\R}J(x-y)v^{\eta}(t,y)dy=\int_{\R}\frac{J(x-y+\eta)-J(x-y)}{\eta}u(t,y)dy.
\end{split}
\end{equation*}
Hence, for any fixed $x$, treating \eqref{lipschitz-eqn-987654} as an ODE in the variable $t$, we find from the variation of constants formula that for any $t\geq t_{0}$
\begin{equation}\label{diff-integral-sol}
\begin{split}
v^{\eta}(t,x)&=v^{\eta}(t_{0},x)e^{-\int_{t_{0}}^{t}(1-a^{\eta}(s,x))ds}+\int_{t_{0}}^{t}b^{\eta}(\tau,x)e^{-\int_{\tau}^{t}(1-a^{\eta}(s,x))ds}d\tau\\
&\quad\quad+\int_{t_{0}}^{t}\tilde{a}^{\eta}(\tau,x)e^{-\int_{\tau}^{t}(1-a^{\eta}(s,x))ds}d\tau.
\end{split}
\end{equation}

Moreover, we set
\begin{equation*}
L_{0}=1+\de_{0}+\sup_{t\in\R}|X(t)-X_{\theta_{0}}^{+}(t)|\quad\text{and}\quad L_{1}=1+\de_{0}+\sup_{t\in\R}|X(t)-X_{\theta_{1}}^{-}(t)|,
\end{equation*}
where $\theta_{1}$ and $\theta_{0}$ are as in (H2) and (H3), respectively. By Lemma \ref{lem-bounded-interface-width}, $L_{0}<\infty$ and $L_{1}<\infty$. We also set
\begin{equation*}
\begin{split}
I_{l}(t)&=(-\infty,X(t)-L_{1}),\\
I_{m}(t)&=[X(t)-L_{1},X(t)+L_{0}],\\
 I_{r}(t)&=(X(t)+L_{0},\infty)
\end{split}
\end{equation*}
for $t\in\R$. Clearly, $I_{l}(t)$, $I_{m}(t)$ and $I_{r}(t)$ are disjoint and $I_{l}(t)\cup I_{m}(t)\cup I_{r}(t)=\R$. Since $X(t)$ may jump, so do $I_{l}(t)$, $I_{m}(t)$ and $I_{r}(t)$.

Since $X(t)\to\pm\infty$ as $t\to\pm\infty$ by Corollary \ref{cor-propagation-estimate}, for any fixed $x\in\R$, there hold $x\in I_{r}(t)$ for all $t\ll-1$ and $x\in I_{l}(t)$ for all $t\gg1$. Thus, for any fixed $x\in\R$,
\begin{equation*}
\begin{split}
t_{\rm first}(x)&=\sup\big\{\tilde{t}\in\R\big|x\in I_{r}(t)\,\,\text{for all}\,\,t\leq \tilde{t}\big\},\\
t_{\rm last}(x)&=\inf\big\{\tilde{t}\in\R\big|x\in I_{l}(t)\,\,\text{for all}\,\,t\geq \tilde{t}\big\}
\end{split}
\end{equation*}
are well-defined. We see that if $X(t)$ is continuous, then $t_{\rm first}(x)$ and $t_{\rm last}(x)$ are the first time and the last time that $x$ is in $I_{m}(t)$. Clearly, $-\infty<t_{\rm first}(x)\leq t_{\rm last}(x)<\infty$ (notice, $t_{\rm first}(x)=t_{\rm last}(x)$ may happen since $X(t)$ may jump). Although the functions $x\mapsto t_{\rm first}(x)$ and $x\mapsto t_{\rm last}(x)$ are unbounded, their difference is a bounded function as given by

\begin{lem}
There holds
\begin{equation}\label{lipschitz-growth-time}
T:=\sup_{x\in\R}[t_{\rm last}(x)-t_{\rm first}(x)]<\infty.
\end{equation}
\end{lem}
\begin{proof}
To see this, we suppose $t_{\rm first}(x)<t_{\rm last}(x)$. Due to possible jumps, we consider two cases.

If $x\notin I_{r}(t_{\rm first}(x))$, then $x\in I_{l}(t_{\rm first}(x))\cup I_{m}(t_{\rm first}(x))$, that is, $x\leq X(t_{\rm first}(x))+L_{0}$. Thus, for all $t\geq t_{\rm first}(x)+T_{1}+\frac{L_{0}+L_{1}+1}{c_{1}}$, we see from Theorem \ref{lem-propagation-estimate} that
\begin{equation*}
x\leq X(t_{\rm first}(x))+L_{0}\leq X(t)-c_{1}(t-t_{\rm first}(x)-T_{1})+L_{0}\leq X(t)-L_{1}-1.
\end{equation*}
This, implies that $x\in I_{l}(t)$ for all $t\geq t_{\rm first}(x)+T_{1}+\frac{L_{0}+L_{1}+1}{c_{1}}$, and hence, by definition
\begin{equation}\label{growth-time-1}
t_{\rm last}(x)\leq t_{\rm first}(x)+T_{1}+\frac{L_{0}+L_{1}+1}{c_{1}}.
\end{equation}

If $x\in I_{r}(t_{\rm first}(x))$, then we can find a sequence $\{t_{n}\}$ satisfying $t_{n}>t_{\rm first}(x)$, $t_{n}\to t_{\rm first}(x)$ as $n\to\infty$ and $x\notin I_{r}(t_{n}(x))$. Then, similar arguments as in the case $x\notin I_{r}(t_{\rm first}(x))$ lead to $t_{\rm last}(x)\leq t_{n}+T_{1}+\frac{L_{0}+L_{1}+1}{c_{1}}$. Passing to the limit $n\to\infty$, we find \eqref{growth-time-1} again. Hence, we have shown \eqref{lipschitz-growth-time}.
\end{proof}

To control $v^{\eta}(t,x)$, it is crucial to control $1-a^{\eta}(t,x)$, which is achieved by means of $t_{\rm first}(x)$ and $t_{\rm last}(x)$ and is given in the following

\begin{lem}
For any $(t,x)\in\R\times\R$ and $0<|\eta|\leq\de_{0}$, there holds
\begin{equation}\label{lipschitz-coefficient-1}
a^{\eta}(t,x)\leq\begin{cases}
1-\ka_{0},&\quad t< t_{\rm first}(x),\\
0,&\quad t> t_{\rm last}(x).
\end{cases}
\end{equation}
\end{lem}
\begin{proof}
By the definition of $t_{\rm first}(x)$ and $t_{\rm last}(x)$, we see
\begin{equation}\label{result-from-defn}
x\in\begin{cases}
I_{r}(t),&\quad t< t_{\rm first}(x),\\
I_{l}(t),&\quad t> t_{\rm last}(x).
\end{cases}
\end{equation}
By (H2), (H3) and the choices of $L_{0}$ and $L_{1}$, we find that for any $(t,x)\in\R\times\R$ and $0<|\eta|\leq\de_{0}$
\begin{equation*}
a^{\eta}(t,x)\leq\begin{cases}
1-\ka_{0},&\quad x\in I_{r}(t),\\
0,&\quad x\in I_{l}(t).
\end{cases}
\end{equation*}
The lemma then follows from \eqref{result-from-defn}.
\end{proof}

The above lemma says that for any fixed $x$, the solution $v^{\eta}(t,x)$ of the ODE \eqref{lipschitz-eqn-987654} can only grow for a period of time that is not longer than $T$.

In the next lemma, we show that $u(t,x)$ is continuous in space.

\begin{lem}\label{lem-tf-continuity}
For any $t\in\R$, $u(t,x)$ is continuous in $x$. Moreover, there holds
\begin{equation*}
\sup_{t\in\R}\sup_{x\neq y}\bigg|\frac{u(t,x)-u(t,y)}{x-y}\bigg|<\infty.
\end{equation*}
\end{lem}
\begin{proof}
Since $0<u(t,x)<1$, we trivially have
\begin{equation*}
\sup_{t\in\R}\sup_{|x-y|\geq\de_{0}}\bigg|\frac{u(t,x)-u(t,y)}{x-y}\bigg|<\infty.
\end{equation*}
Thus, we only need to show
\begin{equation}\label{local-lipschitz-uniform-190}
\sup_{t\in\R}\sup_{0<|x-y|\leq\de_{0}}\bigg|\frac{u(t,x)-u(t,y)}{x-y}\bigg|<\infty.
\end{equation}
To do so, we consider \eqref{diff-integral-sol} for some fixed $x\in\R$. We are going to take $t_{0}\to-\infty$ along some subsequence, and so $t_{0}\ll t_{\rm first}(x)$. For $t$, there are  three cases: $t< t_{\rm first}(x)$, $t\in[t_{\rm first}(x),t_{\rm last}(x)]$ and $t> t_{\rm last}(x)$. Here, we only consider the case $t> t_{\rm last}(x)$; other two cases can be treated similarly and are simpler.

We first note
\begin{equation*}
M_{0}:=\sup_{(t,x)\in\R\times\R}\sup_{0<|\eta|\leq\de_{0}}\Big[|b^{\eta}(t,x)|+|\tilde{a}^{\eta}(t,x)|\Big]<\infty
\end{equation*}
and the following uniform-in-$\eta$ estimates hold:
\begin{equation}\label{lipschitz-estimates}
\begin{split}
e^{-\int_{r}^{t_{\rm first}(x)}(1-a^{\eta}(s,x))ds}&\leq e^{-\ka_{0}(t_{\rm first}(x)-r)},\quad r\leq t_{\rm first}(x),\\
e^{-\int_{r}^{t_{\rm last}(x)}(1-a^{\eta}(s,x))ds}&\leq e^{C_{a}T},\quad r\in[t_{\rm first}(x),t_{\rm last}(x)],\\
e^{-\int_{r}^{t}(1-a^{\eta}(s,x))ds}&\leq e^{-(t-r)},\quad r\in[t_{\rm last}(x),t],
\end{split}
\end{equation}
where $C_{a}:=\sup_{(t,x)\in\R\times\R}\sup_{0<|\eta|\leq\de_{0}}|1-a^{\eta}(t,x)|<\infty$ by (H2). They are simple consequences of \eqref{lipschitz-growth-time} and \eqref{lipschitz-coefficient-1}.

For the second and third terms on the right-hand side of \eqref{diff-integral-sol}, we have
\begin{equation*}
\begin{split}
&\frac{1}{M_{0}}\bigg|\int_{t_{0}}^{t}\big[b^{\eta}(\tau,x)+\tilde{a}^{\eta}(\tau,x)\big]e^{-\int_{\tau}^{t}(1-a^{\eta}(s,x))ds}d\tau\bigg|\leq \int_{t_{0}}^{t}e^{-\int_{\tau}^{t}(1-a^{\eta}(s,x))ds}d\tau\\
&\quad\quad=\int_{t_{0}}^{t_{\rm first}(x)}e^{-\int_{\tau}^{t}(1-a^{\eta}(s,x))ds}d\tau+\int_{t_{\rm first}(x)}^{t_{\rm last}(x)}e^{-\int_{\tau}^{t}(1-a^{\eta}(s,x))ds}d\tau+\int_{t_{\rm last}(x)}^{t}e^{-\int_{\tau}^{t}(1-a^{\eta}(s,x))ds}d\tau.
\end{split}
\end{equation*}
For the three terms on the right-hand side of the above estimate, we use \eqref{lipschitz-estimates} to deduce
\begin{equation*}
\begin{split}
&\int_{t_{0}}^{t_{\rm first}(x)}e^{-\int_{\tau}^{t}(1-a^{\eta}(s,x))ds}d\tau\\
&\quad\quad=\int_{t_{0}}^{t_{\rm first}(x)}e^{-\int_{\tau}^{t_{\rm first}(x)}(1-a^{\eta}(s,x))ds}e^{-\int_{t_{\rm first}(x)}^{t_{\rm last}(x)}(1-a^{\eta}(s,x))ds}e^{-\int_{t_{\rm last}(x)}^{t}(1-a^{\eta}(s,x))ds}d\tau\\
&\quad\quad\leq\int_{t_{0}}^{t_{\rm first}(x)}e^{-\ka_{0}(t_{\rm first}(x)-\tau)}e^{C_{a}T}e^{-(t-t_{\rm last}(x))}d\tau\leq\frac{e^{C_{a}T}}{\ka_{0}},
\end{split}
\end{equation*}

\begin{equation*}
\begin{split}
\int_{t_{\rm first}(x)}^{t_{\rm last}(x)}e^{-\int_{\tau}^{t}(1-a^{\eta}(s,x))ds}d\tau&=\int_{t_{\rm first}(x)}^{t_{\rm last}(x)}e^{-\int_{\tau}^{t_{\rm last}(x)}(1-a^{\eta}(s,x))ds}e^{-\int_{t_{\rm last}(x)}^{t}(1-a^{\eta}(s,x))ds}d\tau\\
&\leq\int_{t_{\rm first}(x)}^{t_{\rm last}(x)}e^{C_{a}T}e^{-(t-t_{\rm last}(x))}d\tau\leq Te^{C_{a}T}
\end{split}
\end{equation*}
and
\begin{equation*}
\begin{split}
\int_{t_{\rm last}(x)}^{t}e^{-\int_{\tau}^{t}(1-a^{\eta}(s,x))ds}d\tau\leq\int_{t_{\rm last}(x)}^{t}e^{-(t-\tau)}d\tau\leq1.
\end{split}
\end{equation*}
Hence, we have shown
\begin{equation}\label{first-estimate-190}
\bigg|\int_{t_{0}}^{t}\big[b^{\eta}(\tau,x)+\tilde{a}^{\eta}(\tau,x)\big]e^{-\int_{\tau}^{t}(1-a^{\eta}(s,x))ds}d\tau\bigg|\leq M_{0}\bigg(\frac{e^{C_{a}T}}{\ka_{0}}+Te^{C_{a}T}+1\bigg).
\end{equation}
Notice the above estimate is universal.

For the first term on the right hand side of \eqref{diff-integral-sol}, we choose $t_{0}$ such that $t_{\rm first}(x)-t_{0}=\frac{1}{|\eta|}$ and claim that
\begin{equation}\label{diff-claim-2}
v^{\eta}(t_{0},x)e^{-\int_{t_{0}}^{t}(1-a^{\eta}(s,x))ds}\to0\quad\text{as}\quad\eta\to0.
\end{equation}
In fact, from $|v^{\eta}(t_{0},x)|\leq\frac{1}{|\eta|}$ and \eqref{lipschitz-estimates}, we see
\begin{equation*}
\begin{split}
\bigg|v^{\eta}(t_{0},x)e^{-\int_{t_{0}}^{t}(1-a^{\eta}(s,x))ds}\bigg|&\leq\frac{1}{|\eta|}e^{-[\int_{t_{0}}^{t_{\rm first}(x)}+\int_{t_{\rm first}(x)}^{t_{\rm last}(x)}+\int_{t_{\rm last}(x)}^{t}](1-a^{\eta}(s,x))ds}\\
&\leq\frac{1}{|\eta|}e^{-\ka_{0}(t_{\rm first}(x)-t_{0})}e^{C_{a}T}e^{-(t-t_{\rm last}(x))}\\
&\leq\frac{1}{|\eta|}e^{-\frac{\ka_{0}}{|\eta|}}e^{C_{a}T}\to0\quad\text{as}\quad\eta\to0.
\end{split}
\end{equation*}

Consequently, choosing $t_{0}$ such that $t_{\rm first}(x)-t_{0}=\frac{1}{|\eta|}$, we conclude from \eqref{diff-integral-sol}, \eqref{first-estimate-190} and \eqref{diff-claim-2} that
\begin{equation*}
\sup_{(t,x)\in\R\times\R}\sup_{0<|\eta|\leq\de_{0}}|v^{\eta}(t,x)|\leq\tilde{C}+M_{0}\bigg(\frac{e^{C_{a}T}}{\ka_{0}}+Te^{C_{a}T}+1\bigg),
\end{equation*}
where $\tilde{C}=\tilde{C}(\ka_{0},\de_{0},T)>0$. This proves \eqref{local-lipschitz-uniform-190}, and hence, the lemma follows.
\end{proof}

We see that for any $(t,x)\in\R\times\R$, as $\eta\to0$, we have
\begin{equation}\label{diff-limits-1}
\begin{split}
&a^{\eta}(t,x)\to f_{u}(t,x,u(t,x)),\\
&\tilde{a}^{\eta}(t,x)\to f_{x}(t,x,u(t,x)),\\
&b^{\eta}(t,x)\to\int_{\R}J'(x-y)u(t,y)dy.
\end{split}
\end{equation}
Notice the first two convergence in \eqref{diff-limits-1} need the continuity of $u(t,x)$ in $x$. By \eqref{diff-limits-1}, \eqref{lipschitz-coefficient-1} holds with $a^{\eta}(t,x)$ replaced by $f_{u}(t,x,u(t,x))$, that is,
\begin{equation}\label{lipschitz-coefficient-2}
f_{u}(t,x,u(t,x))\leq\begin{cases}
1-\ka_{0},&\quad t< t_{\rm first}(x),\\
0,&\quad t> t_{\rm last}(x).
\end{cases}
\end{equation}

Now, we prove Theorem \ref{thm-regularity-of-tf}.

\begin{proof}[Proof of Theorem \ref{thm-regularity-of-tf}]

Let us consider \eqref{diff-integral-sol}. We are going to prove the existence of the limit $\lim_{\eta\to0}v^{\eta}(t,x)$. As in the proof of Lemma \ref{lem-tf-continuity}, we assume $t_{0}\ll t_{\rm first}(x)$ and $t>t_{\rm last}(x)$ in the rest of the proof. We treat three terms on the right-hand side of \eqref{diff-integral-sol} separately.

For the second term on the right-hand side of \eqref{diff-integral-sol}, we claim
\begin{equation}\label{diff-claim-1}
\begin{split}
&\int_{t_{0}}^{t}b^{\eta}(\tau,x)e^{-\int_{\tau}^{t}(1-a^{\eta}(s,x))ds}d\tau\to\int_{t_{0}}^{t}\bigg(\int_{\R}J'(x-y)u(\tau,y)dy\bigg)e^{-\int_{\tau}^{t}(1-f_{u}(s,x,u(s,x)))ds}d\tau\\
&\quad\quad\quad\quad\quad\quad\quad\quad\quad\quad\quad\quad\quad\quad\quad\quad\quad\quad\quad\quad \text{as}\,\, \eta\to0\,\,\text{ uniformly in}\,\, t_{0}\ll t_{\rm first}(x).
\end{split}
\end{equation}
To see this, we notice
\begin{equation*}
\begin{split}
&\bigg|\int_{t_{0}}^{t}b^{\eta}(\tau,x)e^{-\int_{\tau}^{t}(1-a^{\eta}(s,x))ds}d\tau-\int_{t_{0}}^{t}\bigg(\int_{\R}J'(x-y)u(\tau,y)dy\bigg)e^{-\int_{\tau}^{t}(1-f_{u}(s,x,u(s,x)))ds}d\tau\bigg|\\
&\quad\quad\leq\int_{-\infty}^{t}\bigg|b^{\eta}(\tau,x)e^{-\int_{\tau}^{t}(1-a^{\eta}(s,x))ds}-\bigg(\int_{\R}J'(x-y)u(\tau,y)dy\bigg)e^{-\int_{\tau}^{t}(1-f_{u}(s,x,u(s,x)))ds}\bigg|d\tau.
\end{split}
\end{equation*}
By \eqref{diff-limits-1}, the integrand converges to $0$ as $\eta\to0$ pointwise. Thus, by dominated convergence theorem, we only need to make sure that the integrand is controlled by some integrable function that is independent of $\eta$. Writing
\begin{equation*}
b^{0}(\tau,x)=\int_{\R}J'(x-y)u(\tau,y)dy\quad\text{and} \quad a^{0}(s,x)=f_{u}(s,x,u(s,x)),
\end{equation*}
we only need to make sure that the function
\begin{equation*}
\tau\mapsto\sup_{0\leq|\eta|\leq\de_{0}}\bigg|b^{\eta}(\tau,x)e^{-\int_{\tau}^{t}(1-a^{\eta}(s,x))ds}\bigg|
\end{equation*}
is integrable over $(-\infty,t]$.

Setting $M:=\sup_{(t,x)\in\R\times\R}\sup_{0\leq|\eta|\leq\de_{0}}|b^{\eta}(t,x)|<\infty$, we have
\begin{equation*}
\sup_{0\leq|\eta|\leq\de_{0}}\bigg|b^{\eta}(\tau,x)e^{-\int_{\tau}^{t}(1-a^{\eta}(s,x))ds}\bigg|\leq M\sup_{0\leq|\eta|\leq\de_{0}}e^{-\int_{\tau}^{t}(1-a^{\eta}(s,x))ds}.
\end{equation*}
To bound the last integral uniformly in $0\leq|\eta|\leq\de_{0}$, according to \eqref{lipschitz-estimates} and \eqref{lipschitz-coefficient-2}, we consider three cases:

\paragraph{\textbf{Case i. $\tau< t_{\rm first}(x)$}} In this case,
\begin{equation*}
\begin{split}
\sup_{0\leq|\eta|\leq\de_{0}}e^{-\int_{\tau}^{t}(1-a^{\eta}(s,x))ds}&=\sup_{0\leq|\eta|\leq\de_{0}}e^{-[\int_{\tau}^{t_{\rm first}(x)}+\int_{t_{\rm first}(x)}^{t_{\rm last}(x)}+\int_{t_{\rm last}(x)}^{t}](1-a^{\eta}(s,x))ds}\\
&\leq e^{-\ka_{0}(t_{\rm first}(x)-\tau)}e^{C_{a}T}e^{-(t-t_{\rm last}(x))};
\end{split}
\end{equation*}

\paragraph{\textbf{Case ii. $\tau\in[t_{\rm first}(x),t_{\rm last}(x)]$}} In this case,
\begin{equation*}
\sup_{0\leq|\eta|\leq\de_{0}}e^{-\int_{\tau}^{t}(1-a^{\eta}(s,x))ds}=\sup_{0\leq|\eta|\leq\de_{0}}e^{-[\int_{\tau}^{t_{\rm last}(x)}+\int_{t_{\rm last}(x)}^{t}](1-a^{\eta}(s,x))ds}\leq e^{C_{a}T}e^{-(t-t_{\rm last}(x))};
\end{equation*}

\paragraph{\textbf{Case iii. $\tau\in(t_{\rm last}(x),t]$}} In this case, $\sup_{0\leq|\eta|\leq\de_{0}}e^{-\int_{\tau}^{t}(1-a^{\eta}(s,x))ds}\leq e^{-(t-\tau)}$.

Thus, setting
\begin{equation*}
h(\tau)=\begin{cases}
e^{-\ka_{0}(t_{\rm first}(x)-\tau)}e^{C_{a}T}e^{-(t-t_{\rm last}(x))},&\quad\tau< t_{\rm first}(x)\\
e^{C_{a}T}e^{-(t-t_{\rm last}(x))},&\quad\tau\in[t_{\rm first}(x),t_{\rm last}(x)]\\
e^{-(t-\tau)},&\quad\tau\in(t_{\rm last}(x),t],
\end{cases}
\end{equation*}
we find for any $\tau\in(-\infty,t]$
\begin{equation*}
\begin{split}
&\sup_{0<|\eta|\leq\de_{0}}\bigg|b^{\eta}(\tau,x)e^{-\int_{\tau}^{t}(1-a^{\eta}(s,x))ds}-\bigg(\int_{\R}J'(x-y)u(\tau,y)dy\bigg)e^{-\int_{\tau}^{t}(1-f_{u}(s,x,u(s,x)))ds}\bigg|\\
&\quad\quad\leq2\sup_{0\leq|\eta|\leq\de_{0}}\bigg|b^{\eta}(\tau,x)e^{-\int_{\tau}^{t}(1-a^{\eta}(s,x))ds}\bigg|\leq 2h(\tau).
\end{split}
\end{equation*}
To show \eqref{diff-claim-1}, it remains to show $\int_{-\infty}^{t}h(\tau)d\tau<\infty$. But, we readily compute
\begin{equation}\label{diff-uniform-bound}
\begin{split}
\int_{-\infty}^{t}h(\tau)d\tau&=\int_{-\infty}^{t_{\rm first}(x)}h(\tau)d\tau+\int_{t_{\rm first}(x)}^{t_{\rm last}(x)}h(\tau)d\tau+\int_{t_{\rm last}(x)}^{t}h(\tau)d\tau\\
&\leq\int_{-\infty}^{t_{\rm first}(x)}e^{-\ka_{0}(t_{\rm first}(x)-\tau)}e^{C_{a}T}e^{-(t-t_{\rm last}(x))}d\tau\\
&\quad\quad+\int_{t_{\rm first}(x)}^{t_{\rm last}(x)}e^{C_{a}T}e^{-(t-t_{\rm last}(x))}d\tau+\int_{t_{\rm last}(x)}^{t}e^{-(t-\tau)}d\tau\\
&\leq\frac{e^{C_{a}T}}{\ka_{0}}+Te^{C_{a}T}+1.
\end{split}
\end{equation}
Thus, we have shown \eqref{diff-claim-1}. Note that the last bound is uniform in $(t,x)\in\R\times\R$.

For the third term on the right hand side of \eqref{diff-integral-sol}, we claim
\begin{equation}\label{diff-claim-3}
\begin{split}
&\int_{t_{0}}^{t}\tilde{a}^{\eta}(\tau,x)e^{-\int_{\tau}^{t}(1-a^{\eta}(s,x))ds}d\tau\to\int_{t_{0}}^{t}f_{x}(t,x,u(t,x))e^{-\int_{\tau}^{t}(1-f_{u}(s,x,u(s,x)))ds}d\tau\\
&\quad\quad\quad\quad\quad\quad\quad\quad\quad\quad\quad\quad\quad\quad\quad\quad\quad\quad\quad\quad \text{as}\,\, \eta\to0\,\,\text{ uniformly in}\,\, t_{0}\ll t_{\rm first}(x).
\end{split}
\end{equation}
The proof of \eqref{diff-claim-3} is similar to that of \eqref{diff-claim-1}. So, we omit it. Notice
\begin{equation}\label{diff-uniform-bound-1}
\begin{split}
&\int_{-\infty}^{t}\bigg|f_{x}(t,x,u(t,x))e^{-\int_{\tau}^{t}(1-f_{u}(s,x,u(s,x)))ds}\bigg|d\tau\\
&\quad\quad\leq\bigg[\sup_{(t,x,u)\in\R\times\R\times[0,1]}|f_{x}(t,x,u)|\bigg]\int_{-\infty}^{t}h(\tau)d\tau\\
&\quad\quad\leq\bigg[\sup_{(t,x,u)\in\R\times\R\times[0,1]}|f_{x}(t,x,u)|\bigg]\bigg(\frac{e^{C_{a}T}}{\ka_{0}}+Te^{C_{a}T}+1\bigg).
\end{split}
\end{equation}

For the first term on the right hand side of \eqref{diff-integral-sol}, we have \eqref{diff-claim-2}, that is,
\begin{equation}\label{diff-claim-4}
\text{with}\quad t_{\rm first}(x)-t_{0}=\frac{1}{|\eta|},\quad v^{\eta}(t_{0},x)e^{-\int_{t_{0}}^{t}(1-a^{\eta}(s,x))ds}\to0\quad\text{as}\quad\eta\to0.
\end{equation}

Hence, choosing $t_{0}$ such that $t_{\rm first}(x)-t_{0}=\frac{1}{|\eta|}$ and passing to the limit $\eta\to0$ in \eqref{diff-integral-sol}, we conclude from \eqref{diff-claim-1}, \eqref{diff-claim-3} and \eqref{diff-claim-4} that
\begin{equation}\label{a-formula-for-ux}
\begin{split}
u_{x}(t,x)&=\lim_{\eta\to0}v^{\eta}(t,x)\\
&=\int_{-\infty}^{t}\bigg[\int_{\R}J'(x-y)u(\tau,y)dy+f_{x}(\tau,x,u(\tau,x))\bigg]e^{-\int_{\tau}^{t}(1-f_{u}(s,x,u(s,x)))ds}d\tau.
\end{split}
\end{equation}
From which, we see that $u_{x}(t,x)$ is continuous in $(t,x)\in\R\times\R$. Moreover, by \eqref{diff-uniform-bound}, \eqref{diff-uniform-bound-1} and \eqref{a-formula-for-ux}, we have $\sup_{(t,x)\in\R\times\R}|u_{x}(t,x)|<\infty$. This completes the proof.
\end{proof}

\begin{rem}\label{rem-a-formula-for-ux}
From \eqref{a-formula-for-ux}, \eqref{diff-uniform-bound} and \eqref{diff-uniform-bound-1}, we see
\begin{equation*}
\sup_{(t,x)\in\R\times\R}|u_{x}(t,x)|\leq\bigg[\|J'\|_{L^{1}(\R)}+\sup_{(t,x,u)\in\R\times\R\times[0,1]}|f_{x}(t,x,u)|\bigg]\bigg(\frac{e^{C_{a}T}}{\ka_{0}}+Te^{C_{a}T}+1\bigg),
\end{equation*}
where $C_{a}$ depends only on $\sup_{(t,x,u)\in\R\times\R\times[0,1]}|f_{u}(t,x,u)|$ and $T$ is controlled by \eqref{growth-time-1}, and hence, $T$ depends only on $f_{B}$ and the shape of $u(t,x)$.
\end{rem}


\section{Proof of Theorem \ref{thm-regularity-of-tf-ii}}\label{sec-proof-2}

This whole section is devoted to the proof of Theorem \ref{thm-regularity-of-tf-ii}. Let $u(t,x)$ be a transition front as in the statement of Theorem \ref{thm-regularity-of-tf-ii}. Hence, there exist $C>0$ and $r>0$ such that
\begin{equation}\label{harnack-type-inequality}
u(t,x)\leq Ce^{r|x-y|}u(t,y),\quad (t,x,y)\in\R\times\R\times\R.
\end{equation}

Let $X(t)$ and $X^{\pm}_{\la}(t)$ be interface location functions of $u(t,x)$, where $X^{\pm}_{\la}(t)$ are given in \eqref{defn-interface-locations}. As in \cite{LiZl14}, for $(t,x)\in\R\times\R$ and $\eta\in\R$ with $0<|\eta|\leq\de_{0}\ll1$, we consider
\begin{equation*}
v^{\eta}(t,x)=\frac{u(t,x+\eta)-u(t,x)}{\eta}\quad\text{and}\quad w^{\eta}(t,x)=\frac{v^{\eta}(t,x)}{u(t,x)}.
\end{equation*}
We readily check $w^{\eta}(t,x)$ satisfies
\begin{equation}\label{eqn-9182}
w^{\eta}_{t}=a_{1}^{\eta}+a_{2}^{\eta}w^{\eta},
\end{equation}
where $a_{1}^{\eta}=\frac{J\ast v^{\eta}}{u}+\frac{\tilde{a}^{\eta}}{u}$ and $a_{2}^{\eta}=-\frac{J\ast u}{u}+a^{\eta}-\frac{f}{u}$ with
\begin{equation}\label{1-term}
\tilde{a}^{\eta}=\frac{f(t,x+\eta,u(t,x+\eta))-f(t,x,u(t,x+\eta))}{\eta}
\end{equation}
and
\begin{equation}\label{2-term}
a^{\eta}=\frac{f(t,x,u(t,x+\eta))-f(t,x,u(t,x))}{u(t,x+\eta)-u(t,x)}.
\end{equation}

To bound the solution of \eqref{eqn-9182}, we first analyze $a_{1}^{\eta}$ and $a_{2}^{\eta}$. For $a^{\eta}_{1}$, we first see
\begin{equation*}
\bigg|\frac{J\ast v^{\eta}}{u}\bigg|=\frac{1}{u(t,x)}\bigg|\int_{\R}\frac{J(x-y+\eta)-J(x-y)}{\eta}u(t,y)dy\bigg|\leq C\int_{\R}\bigg|\frac{J(x+\eta)-J(x)}{\eta}\bigg|e^{r|x|}dx,
\end{equation*}
where we used \eqref{harnack-type-inequality}. Next, setting $C_{1}:=\sup_{(t,x,u)\in\R\times\R\times[0,1]}|f_{xu}(t,x,u)|$, we have
\begin{equation*}
\begin{split}
\bigg|\frac{\tilde{a}^{\eta}}{u}\bigg|\leq\frac{|f_{x}(t,x+\eta_{*},u(t,x+\eta))|}{u(t,x)}\leq C_{1}\frac{u(t,x+\eta)}{u(t,x)}\leq C_{1}Ce^{r|\eta|},
\end{split}
\end{equation*}
where we used Taylor expansion, the fact $f_{x}(t,x,0)=0$ and \eqref{harnack-type-inequality}. Hence,
\begin{equation}\label{a-1}
C_{2}:=\sup_{(t,x)\in\R\times\R}\sup_{0<|\eta|\leq\de_{0}}|a_{1}^{\eta}|<\infty.
\end{equation}

For $a_{2}^{\eta}$, we first see from \eqref{harnack-type-inequality} that
\begin{equation*}
\frac{1}{C}\int_{\R}J(x)e^{-r|x|}dx\leq\frac{J\ast u}{u}\leq C\int_{\R}J(x)e^{r|x|}dx,
\end{equation*}
and thus, setting $C_{3}:=\frac{1}{C}\int_{\R}J(x)e^{-r|x|}dx$ and $C_{4}:=C\int_{\R}J(x)e^{r|x|}dx$, we find
\begin{equation*}
-C_{4}\leq-\frac{J\ast u}{u}\leq-C_{3}.
\end{equation*}
To control the term $a^{\eta}-\frac{f}{u}$ in $a_{2}^{\eta}$, we set
\begin{equation*}
\tilde{\theta}_{0}:=\min\bigg\{\frac{\theta_{1}}{2},\frac{C_{3}}{2}\Big[\sup_{(t,x,u)\in\R\times\R\times[0,1]}|f(t,x,u)|\Big]^{-1}\bigg\},
\end{equation*}
and define
\begin{equation*}
L_{0}=1+\de_{0}+\sup_{t\in\R}|X(t)-X^{+}_{\tilde{\theta}_{0}}(t)|\quad\text{ and}\quad L_{1}=1+\de_{0}+\sup_{t\in\R}|X(t)-X^{-}_{\theta_{1}}(t)|.
\end{equation*}
As in the proof of Theorem \ref{thm-regularity-of-tf}, we set
\begin{equation*}
\begin{split}
I_{l}(t)&=(-\infty,X(t)-L_{1}),\\
I_{m}(t)&=[X(t)-L_{1},X(t)+L_{0}],\\
 I_{r}(t)&=(X(t)+L_{0},\infty)
\end{split}
\end{equation*}
for $t\in\R$, and for any fixed $x\in\R$, define
\begin{equation*}
\begin{split}
t_{\rm first}(x)&=\sup\{\tilde{t}\in\R|x\in I_{r}(t)\,\,\text{for all}\,\,t\leq \tilde{t}\},\\
t_{\rm last}(x)&=\inf\{\tilde{t}\in\R|x\in I_{l}(t)\,\,\text{for all}\,\,t\geq \tilde{t}\}.
\end{split}
\end{equation*}
Then, there hold $T:=\sup_{x\in\R}[t_{\rm last}(x)-t_{\rm first}(x)]<\infty$ and for all $x\in\R$
\begin{equation*}\label{result-from-defn-2}
x\in\begin{cases}
I_{r}(t),&\quad t< t_{\rm first}(x),\\
I_{l}(t),&\quad t> t_{\rm last}(x).
\end{cases}
\end{equation*}

Now, for $0<|\eta|\leq\de_{0}$, we have
\begin{itemize}
\item if $t< t_{\rm first}(x)$, then $x\in I_{r}(t)$, in particular, $x\geq X^{+}_{\tilde{\theta}_{0}}(t)+\de_{0}$, and hence, $u(t,x)\leq\tilde{\theta}_{0}$ and $u(t,x+\eta)\leq\tilde{\theta}_{0}$; it then follows from Taylor expansion that
\begin{equation*}
\bigg|a^{\eta}-\frac{f}{u}\bigg|=|f_{u}(t,x,u_{*})-f_{u}(t,x,u_{**})|\leq|u_{*}-u_{**}|\sup_{(t,x,u)\in\R\times\R\times[0,1]}|f_{uu}(t,x,u)|\leq\frac{C_{3}}{2},
\end{equation*}
where $u_{*}$ is between $u(t,x)$ and $u(t,x+\eta)$, and $u_{**}$ is between $0$ and $u(t,x)$, and hence, both $u_{*}$ and $u_{**}$ are between $0$ and $\tilde{\theta}_{0}$, so $|u_{*}-u_{**}|\leq\tilde{\theta}_{0}$;

\item if $t> t_{\rm last}(x)$, then $x\in I_{l}(x)$, in particular, $x\leq X^{-}_{\theta_{1}}(t)-\de$, and hence, $u(t,x)\geq\theta_{1}$ and $u(t,x+\eta)\geq\theta_{1}$; it then follows from (H2) that $a^{\eta}\leq0$, which leads to $a^{\eta}-\frac{f}{u}\leq0$;

\item if $t\in[t_{\rm first}(x),t_{\rm last}(x)]$, then $|a^{\eta}-\frac{f}{u}|\leq2\sup_{(t,x,u)\in\R\times\R\times[0,1]}|f_{u}(t,x,u)|$.
\end{itemize}
Therefore, we have the following for $a_{2}^{\eta}$: for $0<|\eta|\leq\de_{0}$ and $x\in\R$
\begin{equation}\label{a-2}
a_{2}^{\eta}\leq\begin{cases}
-\frac{C_{3}}{2},\quad& t\leq t_{\rm first}(x),\\
C_{5},\quad & t_{\rm first}(x)\leq t\leq t_{\rm last}(x),\\
-C_{3},\quad & t\geq t_{\rm last}(x),
\end{cases}
\end{equation}
where $C_{5}:=C_{4}+2\sup_{(t,x,u)\in\R\times\R\times[0,1]}|f_{u}(t,x,u)|$.

With the help of \eqref{a-1} and \eqref{a-2}, we are now able to bound the solution of \eqref{eqn-9182}. Notice, the solution of \eqref{eqn-9182} can be written as
\begin{equation}\label{eqn-integral-9812}
w^{\eta}(t,x)=e^{\int_{t_{0}}^{t}a_{2}^{\eta}(s,x)ds}w^{\eta}(t_{0},x)+\int_{t_{0}}^{t}e^{\int_{\tau}^{t}a_{2}^{\eta}(s,x)ds}a_{1}^{\eta}(\tau,x)d\tau,\quad x\in\R,\quad t\geq t_{0}.
\end{equation}

Using \eqref{a-1}, \eqref{a-2} and \eqref{eqn-integral-9812}, we first argue as in the proof of Lemma \ref{lem-tf-continuity} to conclude that
\begin{equation*}
\sup_{(t,x)\in\R\times\R}\sup_{0<|\eta|\leq\de_{0}}\bigg|\frac{u(t,x+\eta)-u(t,x)}{\eta u(t,x)}\bigg|=\sup_{(t,x)\in\R\times\R}\sup_{0<|\eta|\leq\de_{0}}|w^{\eta}(t,x)|<\infty,
\end{equation*}
which in particular implies that $u(t,x)$ is continuous in $x$, since $u(t,x)\in(0,1)$.

Now, we set $a^{0}_{1}=\frac{J'\ast u}{u}+f_{x}$ and $a_{2}^{0}=-\frac{J\ast u}{u}+f_{u}-\frac{f}{u}$. Using the continuity of $u(t,x)$ in $x$, we have the pointwise limits $a^{\eta}_{1}\to a^{0}_{1}$
and $a^{\eta}_{2}\to a^{0}_{2}$ as $\eta\to0$. Then, using \eqref{a-1} and \eqref{a-2}, we can show that as \eqref{diff-claim-1},
\begin{equation*}
\int_{t_{0}}^{t}e^{\int_{\tau}^{t}a_{2}^{\eta}(s,x)ds}a_{1}^{\eta}(\tau,x)d\tau\to\int_{t_{0}}^{t}e^{\int_{\tau}^{t}a_{2}^{0}(s,x)ds}a_{1}^{0}(\tau,x)d\tau
\end{equation*}
uniformly in $t_{0}\ll t_{\rm first}(x)$ as $\eta\to0$, and as \eqref{diff-claim-2},
\begin{equation*}
\text{with}\,\,t_{\rm first}(x)-t_{0}=\frac{1}{|\eta|},\quad e^{\int_{t_{0}}^{t}a_{2}^{\eta}(s,x)ds}w^{\eta}(t_{0},x)\to0\quad\text{as}\quad\eta\to0.
\end{equation*}
Hence, setting $t_{\rm first}(x)-t_{0}=\frac{1}{|\eta|}$ in \eqref{eqn-integral-9812} and passing to the limit $\eta\to0$, we find
\begin{equation*}
\frac{u_{x}(t,x)}{u(t,x)}=\lim_{\eta\to0}w^{\eta}(t,x)=\int_{-\infty}^{t}e^{\int_{\tau}^{t}a_{2}^{0}(s,x)ds}a_{1}^{0}(\tau,x)d\tau,\quad (t,x)\in\R\times\R.
\end{equation*}
This completes the proof.


\section{Proof of Theorem \ref{thm-regularity-of-tf-iii}}\label{sec-proof-iafjiaf}

We prove Theorem \ref{thm-regularity-of-tf-iii} in this section. Throughout this section, we assume (H1) and (H2). To prove Theorem \ref{thm-regularity-of-tf-iii}, we need the following three lemmas. For $r>0$, let
\begin{equation*}
\Ga_{r}(x)=\min\{1,e^{-rx}\}=\begin{cases}
1,\quad&x\leq0,\\
e^{-rx},\quad&x\geq0.
\end{cases}
\end{equation*}

\begin{lem}\label{lem-090-1}
There exist two continuous functions $M:(0,\infty)\to(0,\infty)$ and $\ga:(0,\infty)\to(0,\infty)$ with $\ga(r)\to0$ such that
\begin{equation*}
\bigg|\frac{[J\ast\Ga_{r}](x)}{\Ga_{r}(x)}-1\bigg|\leq\ga(r),\quad x\geq M(r)
\end{equation*}
for all $r>0$.
\end{lem}
\begin{proof}
Fix $r>0$ and write $\Ga=\Ga_{r}$. We see
\begin{equation*}
\begin{split}
\frac{[J\ast\Ga](x)}{\Ga(x)}-1&=\int_{-\infty}^{0}J(x-y)\frac{\Ga(y)}{\Ga(x)}dy+\int_{0}^{\infty}J(x-y)\frac{\Ga(y)}{\Ga(x)}dy-1\\
&=\int_{-\infty}^{0}J(x-y)\frac{\Ga(y)}{\Ga(x)}dy+\int_{0}^{\infty}J(x-y)e^{r(x-y)}dy-1\\
&=\int_{-\infty}^{0}J(x-y)\frac{\Ga(y)}{\Ga(x)}dy+\bigg[\int_{\R}J(y)e^{ry}dy-1\bigg]-\int_{x}^{\infty}J(y)e^{ry}dy.
\end{split}
\end{equation*}
Due to the decay of $J$ at $\pm\infty$, it is not hard to see that $\lim_{x\to\infty}\int_{-\infty}^{0}J(x-y)\frac{\Ga(y)}{\Ga(x)}dy=0$. In fact, for all large $x$,
\begin{equation*}
\int_{-\infty}^{0}J(x-y)\frac{\Ga(y)}{\Ga(x)}dy=\int_{-\infty}^{0}J(x-y)e^{rx}dy\leq\int_{-\infty}^{0}J(x-y)e^{r(x-y)}dy\to0\,\,\text{as}\,\,x\to\infty.
\end{equation*}
Since $\int_{\R}J(y)dy=1$, we find $\lim_{r\to0}\int_{\R}J(y)e^{ry}dy=1$ by dominated convergence theorem. Clearly, $\lim_{x\to\infty}\int_{x}^{\infty}J(y)e^{ry}dy=0$. The lemma then follows.
\end{proof}

\begin{lem}\label{lem-090-2}
There exists $r_{0}>0$ such that if $r\in(0,r_{0}]$ is such that
\begin{equation}\label{exact-decay-condition}
\lim_{x\to\infty}\frac{u(t,x+X(t))}{e^{-rx}}=1\quad\text{uniformly in}\quad t\in\R,
\end{equation}
then, there exists $M(r)>0$ such that
\begin{equation}\label{result-1}
\frac{1}{2}\leq\frac{[J\ast u(t,\cdot+X(t))](x)}{u(t,x+X(t))}\leq\frac{3}{2},\quad x\geq M(r)\quad \text{and}\quad t\in\R,
\end{equation}
and
\begin{equation}\label{result-2}
\frac{|[J'\ast u(t,\cdot+X(t))](x)|}{u(t,x+X(t))}\leq2,\quad x\geq M(r)\quad \text{and}\quad t\in\R.
\end{equation}
\end{lem}
\begin{proof}
Let $u=u(t,x+X(t))$. For $r\in(0,r_{0}]$, where $r_{0}$ is to be chosen, we let $\Ga=\Ga_{r}$. Write
\begin{equation*}
\begin{split}
\frac{J\ast u}{u}-1&=\frac{J\ast\Ga}{\Ga}\bigg[\frac{J\ast(u-\Ga)}{J\ast\Ga}+1\bigg]\frac{\Ga}{u}-1\\
&=\bigg[\frac{J\ast\Ga}{\Ga}-1\bigg]\bigg[\frac{J\ast(u-\Ga)}{J\ast\Ga}+1\bigg]\frac{\Ga}{u}+\frac{J\ast(u-\Ga)}{J\ast\Ga}\frac{\Ga}{u}+\bigg[\frac{\Ga}{u}-1\bigg].
\end{split}
\end{equation*}
By Lemma \ref{lem-090-1} and \eqref{exact-decay-condition}, we only need to treat the term $\frac{J\ast(u-\Ga)}{J\ast\Ga}$ for large $x$.

By \eqref{exact-decay-condition}, for any $\ep>0$, there exists $M(\ep,r)>0$ such that
\begin{equation*}
|u(t,x+X(t))-e^{-rx}|\leq\ep e^{-rx},\quad x\geq M(\ep,r)\quad \text{and}\quad t\in\R.
\end{equation*}
Then, for $\ep>0$, we have
\begin{equation*}
\begin{split}
\bigg|\frac{J\ast(u-\Ga)}{J\ast\Ga}\bigg|&=\frac{\Ga}{J\ast\Ga}\bigg|\frac{J\ast(u-\Ga)}{\Ga}\bigg|\\
&\leq\frac{\Ga}{J\ast\Ga}\bigg[\int_{-\infty}^{M(\ep,r)}J(x-y)\frac{|u(t,y+X(t))-\Ga(y)|}{\Ga(x)}dy\\
&\quad\quad\quad\quad\quad+\int_{M(\ep,r)}^{\infty}J(x-y)\frac{|u(t,y+X(t))-e^{-ry}|}{e^{-rx}}dy\bigg]\\
&\leq\frac{\Ga}{J\ast\Ga}\bigg[\int_{-\infty}^{M(\ep,r)}J(x-y)\frac{|u(t,y+X(t))-\Ga(y)|}{\Ga(x)}dy+\ep\int_{M(\ep,r)}^{\infty}J(x-y)e^{r(x-y)}dy\bigg]\\
&\leq\frac{\Ga}{J\ast\Ga}\bigg[\int_{-\infty}^{M(\ep,r)}J(x-y)\frac{|u(t,y+X(t))-\Ga(y)|}{\Ga(x)}dy+\ep\int_{\R}J(y)e^{ry}dy\bigg]
\end{split}
\end{equation*}
Due to the decay of $J$ at $\pm\infty$, we have
\begin{equation*}
\lim_{x\to\infty}\int_{-\infty}^{M(\ep,r)}J(x-y)\frac{|u(t,y+X(t))-\Ga(y)|}{\Ga(x)}dy=0\quad\text{uniformly in}\quad t\in\R.
\end{equation*}
It then follows from Lemma \ref{lem-090-1} that for any $\ep>0$, there exists $\tilde{M}(\ep,r)>0$ such that
\begin{equation*}
\bigg|\frac{J\ast(u-\Ga)}{J\ast\Ga}\bigg|\leq\tilde{\ga}(r)\bigg[\frac{\ep}{2}+\ep\int_{\R}J(y)e^{ry}dy\bigg],\quad x\geq\tilde{M}(\ep,r)\quad\text{and}\quad t\in\R,
\end{equation*}
where $\tilde{\ga}(r)\to0$ as $r\to0$. The result \eqref{result-1} follows by first fixing $\ep$, say $\ep=\frac{1}{100}$, and then choosing $r_{0}$ small so that if $r\in(0,r_{0}]$, then $|\frac{J\ast u}{u}-1|\leq\frac{1}{2}$ for all large $x$ depending on $r$. Note that the term $\frac{J\ast\Ga}{\Ga}-1$ and Lemma \ref{lem-090-1} restrict $r_{0}$. Similar arguments lead to \eqref{result-2}, since
\begin{equation*}
\bigg|\frac{J'\ast u}{u}\bigg|\leq\frac{|J'|\ast\Ga}{\Ga}\bigg[\frac{|J'|\ast(u-\Ga)}{|J'|\ast\Ga}+1\bigg]\frac{\Ga}{u}.
\end{equation*}
This completes the proof.
\end{proof}

\begin{lem}\label{lem-away-from-0}
Let $u(t,x)$ be an arbitrary transition front of \eqref{main-eqn} with interface location function $X(t)$. Then, there holds
\begin{equation*}
\forall L>0,\quad \inf_{t\in\R}\inf_{x\leq L+X(t)}u(t,x)>0.
\end{equation*}
\end{lem}
\begin{proof}
Fix $L>0$ and $x_{*}>0$. Let $\la_{1}\in(\frac{1}{2},1)$ and $\la_{2}\in(0,\frac{1}{10})$. We define $u_{0}^{B}:\R\to[0,1]$ and $u_{0}^{M}:\R\to[0,1]$ by setting
\begin{equation*}
u_{0}^{B}(x)=\begin{cases}
\la_{1},\quad &x\leq-x_{*},\\
-\frac{\la_{1}}{x_{*}}x,\quad& x\in[-x_{*},0],\\
0,\quad &x\geq0
\end{cases}
\quad\text{and}\quad
u_{0}^{M}(x)=\begin{cases}
1,\quad &x\leq0,\\
\frac{\la_{2}-1}{x_{*}}x+1,\quad&x\in[0,x_{*}],\\
\la_{2},\quad &x\geq x_{*}.
\end{cases}
\end{equation*}
Clearly, $u_{0}^{B}(\cdot-X^{-}_{\la_{1}}(t))\leq u(t,\cdot)\leq u_{0}^{M}(\cdot-X^{+}_{\la_{2}}(t))$ for all $t\in\R$. Now, denote by $u_{B}(t,x;u_{0}^{B})$ and $u_{M}(t,x;u_{0}^{M})$ the solutions of $u_{t}=J\ast u-u+f_{B}(u)$ and $u_{t}=J\ast u-u+f_{M}(u)$, respectively, with initial data $u_{B}(0,\cdot;u_{0}^{B})=u_{0}^{B}$ and $u_{M}(0,\cdot;u_{0}^{M})=u_{0}^{M}$. It then follows from comparison principle and homogeneity that
\begin{equation}\label{some-result-8472828373}
u_{B}(T,x-X^{-}_{\la_{1}}(t-T);u_{0}^{B})\leq u(t,x)\leq u_{M}(T,x-X^{+}_{\la_{2}}(t-T);u_{0}^{M}),\quad (t,x)\in\R\times\R
\end{equation}
for all $T\geq0$. Now, we consider a small $T$ and let $\la>0$ be small. Let $\xi_{\la}^{B}(T)$ be such that $u_{B}(T,\xi_{\la}^{B}(T);u_{0}^{B})=\la$. We then see from the first inequality in \eqref{some-result-8472828373} that if $x\leq\xi^{B}_{\la}(T)+X^{-}_{\la_{1}}(t-T)-1$, then the monotonicity of $u_{B}(t,x;u_{0}^{B})$ in $x$ yields
\begin{equation*}
u(t,x)\geq u_{B}(T,\xi_{\la}^{B}(T)-1;u_{0}^{B})>u_{B}(T,\xi_{\la}^{B}(T);u_{0}^{B})=\la,
\end{equation*}
which then leads to
\begin{equation*}
X^{-}_{\la}(t)\geq\xi^{B}_{\la}(T)+X^{-}_{\la_{1}}(t-T)-1,\quad t\in\R.
\end{equation*}
Note that if we can find some $C>0$ such that
\begin{equation}\label{a-key-step-9828318317}
X^{-}_{\la_{1}}(t-T)\geq X(t)-C,\quad t\in\R,
\end{equation}
then we can make $\la$ closer to $0$ such that $\xi^{B}_{\la}(T)$ is so large that $X_{\la}^{-}(t)\geq L+X(t)+1$ for all $t\in\R$, which then leads to
\begin{equation*}
\inf_{t\in\R}\inf_{x\leq L+X(t)}u(t,x)\geq \inf_{t\in\R}\inf_{x\leq X_{\la}^{-}(t)-1}u(t,x)>\la>0.
\end{equation*}
Hence, to finish the proof, we only need to show \eqref{a-key-step-9828318317}.

We now show \eqref{a-key-step-9828318317}. Let us look at the interface locations for $u(t,x)$ and $u_{M}(t,x;u_{0}^{M})$ at $\frac{1}{2}$. From the second inequality in \eqref{some-result-8472828373}, we see
\begin{equation}\label{some-inequalities-198314717419741}
 X^{+}_{\frac{1}{2}}(t)\leq \xi^{M}_{\frac{1}{2}}(T)+X^{+}_{\la_{2}}(t-T)+1,\quad t\in\R,
\end{equation}
where $\xi^{M}_{\frac{1}{2}}(T)$ is such that $u(T,\xi^{M}_{\frac{1}{2}}(T);u_{0}^{M})=\frac{1}{2}$. Notice choosing $T$ or $\la_{2}$ smaller, we can guarantee that $\xi^{M}_{\frac{1}{2}}(T)$ is well-defined. We then deduce from \eqref{some-inequalities-198314717419741} that
\begin{equation*}
\begin{split}
X_{\la_{1}}^{-}(t-T)&\geq X_{\la_{2}}^{+}(t-T)-C_{1}\\
&\geq X^{+}_{\frac{1}{2}}(t)-\xi_{\frac{1}{2}}^{M}(T)-1-C_{1}\\
&\geq X(t)-C_{2}-\xi_{\frac{1}{2}}^{M}(T)-1-C_{1}
\end{split}
\end{equation*}
for all $t\in\R$, where $C_{1}=\sup_{t\in\R}|X_{\la_{1}}^{-}(t)-X_{\la_{2}}^{+}(t)|$ and $C_{2}=\sup_{t\in\R}|X_{\frac{1}{2}}^{+}(t)-X(t)|$. Setting $C=C_{2}+\xi_{\frac{1}{2}}^{M}(T)+1+C_{1}$, we find \eqref{a-key-step-9828318317}, and hence, the lemma follows.
\end{proof}

We are ready to prove Theorem \ref{thm-regularity-of-tf-iii}.

\begin{proof}[Proof of Theorem \ref{thm-regularity-of-tf-iii}]
Let $r_{0}>0$ be as in Lemma \ref{lem-090-2} and fix $r\in(0,r_{0}]$. Let $u(t,x)$ be an arbitrary transition front of \eqref{main-eqn} satisfying
\begin{equation}\label{exact-decay-condition-134141}
\lim_{x\to\infty}\frac{u(t,x+X(t))}{e^{-rx}}=1\quad\text{uniformly in}\quad t\in\R.
\end{equation}

To prove the theorem, we proceed as in the proof of Theorem \ref{thm-regularity-of-tf-ii}. Thus, we only need to bound
$a_{1}^{\eta}=\frac{J\ast v^{\eta}}{u}+\frac{\tilde{a}^{\eta}}{u}$
as in \eqref{a-1} and estimate $a_{2}^{\eta}=-\frac{J\ast u}{u}+a^{\eta}-\frac{f}{u}$ as in \eqref{a-2}, where $\tilde{a}^{\eta}$ and $a^{\eta}$ are given in \eqref{1-term} and \eqref{2-term}, respectively.

For $a_{1}^{\eta}$, we have
\begin{equation*}
|a_{1}^{\eta}|\leq\bigg|\frac{J\ast v^{\eta}}{u}\bigg|+\bigg|\frac{\tilde{a}^{\eta}}{u}\bigg|\leq\frac{1}{u(t,x)}\int_{\R}\frac{|J(x-y+\eta)-J(x-y)|}{\eta}u(t,y)dy+C_{1}\frac{u(t,x+\eta)}{u(t,x)},
\end{equation*}
where $C_{1}=\sup_{(t,x,u)\in\R\times\R\times[0,1]}|f_{xu}(t,x,u)|$. For a sufficiently large $M_{1}>0$, we see from \eqref{result-2} and \eqref{exact-decay-condition-134141} that
\begin{equation*}
\sup_{t\in\R}\sup_{x\geq M_{1}+X(t)}\sup_{0<|\eta|\leq\de_{0}}|a_{1}^{\eta}|<\infty.
\end{equation*}
Since $\inf_{t\in\R}\inf_{x\leq M_{1}+X(t)}u(t,x)>0$ by Lemma \ref{lem-away-from-0}, we have
\begin{equation*}
\sup_{t\in\R}\sup_{x\leq M_{1}+X(t)}\sup_{0<|\eta|\leq\de_{0}}|a_{1}^{\eta}|<\infty.
\end{equation*}
Hence, we obtain
\begin{equation*}
\sup_{(t,x)\in\R\times\R}\sup_{0<|\eta|\leq\de_{0}}|a_{1}^{\eta}|<\infty.
\end{equation*}

For $a^{\eta}_{2}$, we first see from \eqref{result-1} that we can find a sufficiently large $M_{2}>0$ such that
\begin{equation*}
\frac{1}{2}\leq\frac{J\ast u}{u}\leq\frac{3}{2},\quad x\geq M_{2}+X(t)\quad\text{and}\quad t\in\R.
\end{equation*}
Since $\inf_{t\in\R}\inf_{x\leq M_{2}+X(t)}u(t,x)>0$ by Lemma \ref{lem-away-from-0}, there exist $C_{1}>0$ and $C_{2}>0$ such that
\begin{equation*}
C_{1}\leq\frac{J\ast u}{u}\leq C_{2},\quad x\leq M_{2}+X(t)\quad\text{and}\quad t\in\R.
\end{equation*}
Then, setting $C_{3}=\min\{\frac{1}{2},C_{1}\}$ and $C_{4}=\max\{\frac{3}{2},C_{2}\}$, we have
\begin{equation*}
-C_{4}\leq-\frac{J\ast u}{u}\leq-C_{3},\quad(t,x)\in\R\times\R.
\end{equation*}
We then follow the arguments as in the proof of Theorem \ref{thm-regularity-of-tf-ii} to conclude an estimate for $a_{2}^{\eta}$ as in \eqref{a-2}.

The rest of the proof can be done along the same line as in the proof of Theorem \ref{thm-regularity-of-tf-ii} and then we complete the proof.
\end{proof}


\section{Proof of Theorem \ref{thm-modified-interface-location} and Corollaries \ref{cor-modified-interface}-\ref{cor-modified-interface-ii}}\label{sec-proof-thm-cor}

In this section, we prove Theorem \ref{thm-modified-interface-location},  Corollary \ref{cor-modified-interface} and Corollary \ref{cor-modified-interface-ii}. We first prove Theorem \ref{thm-modified-interface-location}.

\begin{proof}[Proof of Theorem \ref{thm-modified-interface-location}]
Let $u(t,x)$ be an arbitrary transition front of \eqref{main-eqn} with interface location function $X(t)$ as in the statement of Theorem \ref{thm-modified-interface-location}. Then, by Theorem \ref{lem-propagation-estimate} and the assumption, we have
\begin{equation}\label{lower-upper-estimate}
c_{1}(t-t_{0}-T_{1})\leq X(t)-X(t_{0})\leq c_{2}(t-t_{0}+T_{2}),\quad t\geq t_{0}.
\end{equation}
We modify $X(t)$ within two steps by means of \eqref{lower-upper-estimate}. The first step gives a continuous modification. The second step gives the continuously differentiable modification as in the statement of the theorem. We remark that two inequalities in \eqref{lower-upper-estimate} play different roles in the following arguments. While the first inequality in \eqref{lower-upper-estimate} pushes $X(t)$ move to the right, the second inequality in \eqref{lower-upper-estimate} controls the possible jumps of $X(t)$.

\paragraph{\bf{Step 1.}} We show there is a continuous function $\tilde{X}:\R\to\R$ such that $\sup_{t\in\R}|\tilde{X}(t)-X(t)|<\infty$. Fix some $T>0$. At $t=0$, let
\begin{equation*}
Z^{+}(t;0)=X(0)+c_{2}(T+T_{2})+\frac{c_{1}}{2}t,\quad t\geq0
\end{equation*}
By the second inequality in \eqref{lower-upper-estimate}, $X(t)<Z^{+}(t;0)$ for all $[0,T]$. By the first inequality in \eqref{lower-upper-estimate}, we have $X(t)>Z^{+}(t;0)$ for all large $t$. Define $T_{1}^{+}=\inf\{t\geq0|X(t)\geq Z^{+}(t;0)\}$. By \eqref{lower-upper-estimate}, it is easy to see that $T_{1}^{+}\in[T,\frac{c_{2}(T+T_{2})+c_{1}T_{1}}{c_{1}/2}]$. At the moment $T_{1}^{+}$, $X(t)$ may jump, but, due to the second inequality in \eqref{lower-upper-estimate}, the jump is at most $c_{2}T_{2}$. Thus, we obtain
\begin{equation*}
\begin{split}
X(t)&<Z^{+}(t;0)\,\,\text{for}\,\,t\in[0,T_{1}^{+}),\\
X(T_{1}^{+})&\in[Z^{+}(T_{1}^{+};0)-c_{2}T_{2},Z^{+}(T_{1}^{+};0)+c_{2}T_{2}).
\end{split}
\end{equation*}

Next, at $t=T_{1}^{+}$, let
\begin{equation*}
Z^{+}(t;T_{1}^{+})=X(T_{1}^{+})+c_{2}(T+T_{2})+\frac{c_{1}}{2}(t-T_{1}^{+}),\quad t\geq T_{1}^{+}.
\end{equation*}
Then, $T_{2}^{+}=\inf\{t\geq T_{1}^{+}|X(t)\geq Z^{+}(t;T_{1}^{+})\}$ is well-defined, and $T_{2}^{+}-T_{1}^{+}\in[T,\frac{c_{2}(T+T_{2})+c_{1}T_{1}}{c_{1}/2}]$. Moreover, there hold
\begin{equation*}
\begin{split}
X(t)&<Z^{+}(t;T_{1}^{+})\,\,\text{for}\,\,t\in[T_{1}^{+},T_{2}^{+}),\\ X(T_{2}^{+})&\in[Z^{+}(T_{2}^{+};T_{1}^{+})-c_{2}T_{2},Z^{+}(T_{2}^{+};T_{1}^{+})+c_{2}T_{2}).
\end{split}
\end{equation*}

Repeating the above arguments, we obtain the following, there is a sequence of times $\{T_{n-1}^{+}\}_{n\geq1}$ satisfying $T_{0}^{+}=0$, $T_{n}^{+}-T_{n-1}^{+}\in[T,\frac{c_{2}(T+T_{2})+c_{1}T_{1}}{c_{1}/2}]$  and
\begin{equation}\label{an-estimate-0198932r592}
\begin{split}
X(t)&<Z^{+}(t;T_{n-1}^{+})\,\,\text{for}\,\,t\in[T_{n-1}^{+},T_{n}^{+}),\\
X(T_{n}^{+})&\in[Z^{+}(T_{n}^{+};T_{n-1}^{+}),Z^{+}(T_{n}^{+};T_{n-1}^{+})+c_{2}T_{2}),
\end{split}
\end{equation}
for all $n\geq1$, where $Z^{+}(t;T_{n-1}^{+})=X(T_{n-1}^{+})+c_{2}(T+T_{2})+\frac{c_{1}}{2}(t-T_{n-1}^{+})$.

We define $Z^{+}:[0,\infty)\to\R$ by setting
\begin{equation*}
Z^{+}(t)=Z^{+}(t;T_{n-1}^{+}),\quad t\in[T^{+}_{n-1},T^{+}_{n}),\quad n\geq1
\end{equation*}
Since $\sup_{n\geq1}[T^{+}_{n-1},T^{+}_{n})=[0,\infty)$, $Z^{+}(t)$ is well-defined. It follows from \eqref{an-estimate-0198932r592} that $X(t)<Z^{+}(t)$ for all $t\geq0$. Moreover, for $t\in[T^{+}_{n-1},T^{+}_{n})$,
\begin{equation*}
\begin{split}
Z^{+}(t)-X(t)&\leq X(T_{n-1}^{+})+c_{2}(T+T_{2})+\frac{c_{1}}{2}(t-T_{n-1}^{+})-[X(T_{n-1}^{+})+c_{1}(t-T_{n-1}^{+}-T_{1})]\\
&\leq c_{2}(T+T_{2})-\frac{c_{1}}{2}(t-T_{n-1}^{+})+c_{1}T_{1}\leq c_{2}(T+T_{2})+c_{1}T_{1}.
\end{split}
\end{equation*}
Hence, $0<Z^{+}(t)-X(t)\leq c_{2}(T+T_{2})+c_{1}T_{1}$ for all $t\in[0,\infty)$. Modifying $Z^{+}(t)$ near $T_{n-1}^{+}$ for $n\geq1$, we find a continuous function $\tilde{Z}^{+}:[0,\infty)\to\R$ such that $\sup_{t\in[0,\infty)}|\tilde{Z}^{+}(t)-X(t)|<\infty$.

Clearly, we can mimic the above arguments to find a continuous function $\tilde{Z}^{-}:(-\infty,0]\to\R$ such that $\sup_{t\in(-\infty,0]}|\tilde{Z}^{-}(t)-X(t)|<\infty$. Combining $\tilde{Z}^{\pm}(t)$ and modifying near $0$, we find a continuous function $\tilde{X}:\R\to\R$ such that $\sup_{t\in\R}|\tilde{X}(t)-X(t)|<\infty$.

\paragraph{\bf{Step 2.}} By Step 1, we assume, without loss of generality, that $X(t)$ is continuous. We proceed as in Step 1.

Fix any $t_{0}\in\R$ and consider it as an initial moment. At the initial moment $t_{0}$, we define $Z(t;t_{0})=X(t_{0})+C_{0}+\frac{c_{1}}{2}(t-t_{0})$ for $ t\geq t_{0}$, where $C_{0}>0$ is so large that $C_{0}>c_{2}T_{2}$. Clearly, $X(t_{0})<Z(t_{0};t_{0})$. By the first inequality in \eqref{lower-upper-estimate} and continuity, $X(t)$ will hit $Z(t;t_{0})$ sometime after $t_{0}$. Let $T_{1}(t_{0})$ be the first time that $X(t)$ hits $Z(t;t_{0})$, that is, $T_{1}(t_{0})=\min\big\{t\geq t_{0}\big|X(t)=Z(t;t_{0})\big\}$. It follows that $X(t)<Z(t;t_{0})$ for $t\in[t_{0},T_{1}(t_{0}))$ and $X(T_1(t_{0}))=Z(T_{1}(t_{0});t_{0})$. Moreover, $T_{1}(t_{0})-t_{0}\in\Big[\frac{C_{0}-c_{2}T_{2}}{c_{2}-c_{1}/2},\frac{C_{0}+c_{1}T_{1}}{c_{1}/2}\Big]$, which is a simple result of \eqref{lower-upper-estimate}.

Now, at the moment $T_{1}(t_{0})$, we define $Z(t;T_{1}(t_{0}))=X(T_{1}(t_{0}))+C_{0}+\frac{c_{1}}{2}(t-T_{1}(t_{0}))$ for $t\geq T_{1}(t_{0})$. Similarly, $X(T_{1}(t_{0}))<Z(T_{1}(t_{0});T_{1}(t_{0}))$ and $X(t)$ will hit $Z(t;T_{1}(t_{0}))$ sometime after $T_{1}(t_{0})$. Denote by $T_{2}(t_{0})$ the first time that $X(t)$ hits $Z(t;T_{1}(t_{0}))$. Then, $X(t)<Z(t;T_{1}(t_{0}))\,\,\text{for}\,\,t\in[T_{1}(t_{0}),T_{2}(t_{0}))$ and $X(T_{2}(t_{0}))=Z(T_{2}(t_{0});T_{1}(t_{0}))$, and $T_{2}(t_{0})-T_{1}(t_{0})\in\Big[\frac{C_{0}-c_{2}T_{2}}{c_{2}-c_{1}/2},\frac{C_{0}+c_{1}T_{1}}{c_{1}/2}\Big]$.

Repeating the above arguments, we obtain the following: there is a sequence of times $\{T_{n-1}(t_{0})\}_{n\in\N}$ satisfying $T_{0}(t_{0})=t_{0}$ and
\begin{equation}\label{uniform-time-interval}
T_{n}(t_{0})-T_{n-1}(t_{0})\in\bigg[\frac{C_{0}-c_{2}T_{2}}{c_{2}-c_{1}/2},\frac{C_{0}+c_{1}T_{1}}{c_{1}/2}\bigg],\quad\forall\,\,n\in\N,
\end{equation}
and  for any $n\in\N$
\begin{equation*}
\begin{split}
X(t)<Z(t;T_{n-1}(t_{0}))\,\,\text{for}\,\,t\in[T_{n-1}(t_{0}),T_{n}(t_{0}))\quad\text{and}\quad X(T_{n}(t_{0}))=Z(T_{n}(t_{0});T_{n-1}(t_{0})),
\end{split}
\end{equation*}
where $Z(t;T_{n-1}(t_{0}))=X(T_{n-1}(t_{0}))+C_{0}+\frac{c_{1}}{2}(t-T_{n-1}(t_{0}))$. Moreover, for any $n\in\N$ and $t\in[T_{n-1}(t_{0}),T_{n}(t_{0}))$, we conclude from \eqref{lower-upper-estimate} that
\begin{equation*}
\begin{split}
&Z(t;T_{n-1}(t_{0}))-X(t)\\
&\quad\quad\leq X(T_{n-1}(t_{0}))+C_{0}+\frac{c_{1}}{2}(t-T_{n-1}(t_{0}))-\big[X(T_{n-1}(t_{0}))+c_{1}(t-T_{n-1}(t_{0})-T_{1})\big]\\
&\quad\quad=C_{0}+c_{1}T_{1}-\frac{c_{1}}{2}(t-T_{n-1}(t_{0}))\leq C_{0}+c_{1}T_{1}.
\end{split}
\end{equation*}

Next, define $\tilde{Z}(\cdot;t_{0}):[t_{0},\infty)\ra\R$ by setting
\begin{equation}\label{definition-new-fun}
\tilde{Z}(t;t_{0})=Z(t;T_{n-1}(t_{0}))\quad\text{for}\,\, t\in[T_{n-1}(t_{0}),T_{n}(t_{0})),\,\,n\in\N.
\end{equation}
Since $[t_{0},\infty)=\cup_{n\in\N}[T_{n-1}(t_{0}),T_{n}(t_{0}))$ by \eqref{uniform-time-interval}, $\tilde{Z}(t;t_{0})$ is well-defined for all $t\geq t_{0}$. Notice $\tilde{Z}(t;t_{0})$ is strictly increasing and is linear on $[T_{n-1}(t_{0}),T_{n}(t_{0}))$ with slope $\frac{c_{1}}{2}$ for each $n\in\N$, and satisfies
\begin{equation*}
0\leq\tilde{Z}(t;t_{0})-X(t)\leq C_{0}+c_{1}T_{1},\quad t\geq t_{0}.
\end{equation*}

Due to \eqref{uniform-time-interval}, we can modify $\tilde{Z}(t;t_{0})$ near each $T_{n}(t_{0})$ for $n\in\N$ as follows. Fix some $\de_{*}\in\Big(0,\frac{1}{2}\frac{C_{0}-c_{2}T_{2}}{c_{2}-c_{1}/2}\Big)$. We modify $\tilde{Z}(t;s)$ by redefining it on the intervals $(T_{n}(t_{0})-\de_{*},T_{n}(t_{0}))$, $n\in\N$ as follows: define
\begin{equation*}
X(t;t_{0})=\begin{cases}
\tilde{Z}(t;t_{0}),\quad t\in[t_{0},\infty)\bs\cup_{n\in\N}(T_{n}(t_{0})-\de_{*},T_{n}(t_{0})),\\
X(T_{n}(t_{0}))+\de(t-T_{n}(t_{0})),\quad t\in(T_{n}(t_{0})-\de_{*},T_{n}(t_{0})),\,\,n\in\N,
\end{cases}
\end{equation*}
where $\de:[-\de_{*},0]\ra[-\frac{1}{2}c_{1}\de_{*},C_{0}]$ is twice continuously differentiable and satisfies
\begin{equation*}
\begin{split}
&\de(-\de_{*})=-\frac{c_{1}}{2}\de_{*},\quad\de(0)=C_{0},\\
&\dot{\de}(-\de_{*})=\frac{c_{1}}{2}=\dot{\de}(0),\quad\dot{\de}(t)\geq\frac{c_{1}}{2}\,\,\text{for}\,\,t\in(-\de_{*},0)\quad\text{and}\\
&\ddot{\de}(-\de_{*})=0=\ddot{\de}(0).
\end{split}
\end{equation*}
The existence of such a function $\de(t)$ is clear. Moreover, there exist $c_{\max}=c_{\max}(\de_{*})>0$ and $\tilde{c}_{\max}=\tilde{c}_{\max}(\de_{*})>0$ such that $\dot{\de}(t)\leq c_{\max}$ and $|\ddot{\de}(t)|\leq\tilde{c}_{\max}$ for $t\in(-\de_{*},0)$. Notice the above modification is independent of $n\in\N$ and $t_{0}$. Hence, $X(t;t_{0})$ satisfies the following uniform in $t_{0}$ properties:
\begin{itemize}
\item $0\leq X(t;t_{0})-X(t)\leq d_{\max}$ for some $d_{\max}>0$,
\item $\frac{c_{1}}{2}\leq\dot{X}(t;t_{0})\leq c_{\max}$,
\item $|\ddot{X}(t;t_{0})|\leq\tilde{c}_{\max}$.
\end{itemize}
Since $X(t)$ is continuous, so locally bounded, we may apply Arzel\`{a}-Ascoli theorem to conclude the existence of some continuously differentiable function $\tilde{X}:\R\to\R$ such that $X(t;t_{0})\to \tilde{X}(t)$ and $\dot{X}(t;t_{0})\to\dot{\tilde{X}}(t)$ locally uniformly in $t$ as $t_{0}\to-\infty$ along some subsequence. It's easy to see that $\tilde{X}(t)$ satisfies all the properties as in the statement of the theorem.
\end{proof}

Next, we prove Corollary \ref{cor-modified-interface}. Recall that for a given transition front $u(t,x)$ of \eqref{main-eqn}, $X^{\pm}_{\la}(t)$ are defined in \eqref{defn-interface-locations}.

\begin{proof}[Proof of Corollary \ref{cor-modified-interface}]
We modify the proof of Theorem \ref{lem-propagation-estimate}. Let $u(t,x)$ be an arbitrary transition front of \eqref{main-eqn} with interface location function $X(t)$. Since $f(t,x,u)\leq0$ for $(t,x,u)\in\R\times\R\times[0,\tilde{\theta}]$, we can find a function $f_{I}(u)$ such that $f(t,x,u)\leq f_{I}(u)$ for $(t,x,u)\in\R\times\R\times[0,1]$, where $f_{I}:[0,1]\to\R$ is $C^{2}$ and is of standard ignition type, that is, there exists $\theta_{I}\in(0,1)$ such that
\begin{equation*}
f_{I}(u)=0,\,\,u\in[0,\theta_{I}]\cup\{1\},\quad f_{I}(u)>0,\,\,u\in(\theta_{I},1)\quad\text{and}\quad f_{I}'(1)<0.
\end{equation*}

Fix some $\la\in(\theta,1)$. We  write $X^{+}(t)=X^{+}_{\la}(t)$. Since $\sup_{t\in\R}|X(t)-X^{+}(t)|<\infty$ by Lemma \ref{lem-bounded-interface-width}, it suffices to show
\begin{equation}\label{propagation-estimate-1-ig}
X^{+}(t)-X^{+}(t_{0})\leq c(t-t_{0}+T),\quad t\geq t_{0}
\end{equation}
for some $c>0$ and $T>0$.

To do so, we fix some $\tilde{\theta}_{I}\in(0,\theta_{I})$. Let $(c_{I},\phi_{I})$ with $c_{I}>0$ be the unique solution of
\begin{equation*}
\begin{cases}
J\ast\phi-\phi+c\phi_{x}+f_{I}(\phi)=0,\\
\phi_{x}<0,\,\,\phi(0)=\theta_{I},\,\,\phi(-\infty)=1\,\,\text{and}\,\,\phi(\infty)=\tilde{\theta}_{I}.
\end{cases}
\end{equation*}
Note that $\phi_{I}$ connects $\tilde{\theta}_{I}$ and $1$ instead of $0$ and $1$ (see Appendix \ref{app-ig-tw} for more properties about $\phi_{I}$; in Appendix \ref{app-ig-tw}, we consider traveling waves connecting $0$ and $1$, but by simple shift, all results there apply here).

Let $u_{0}:\R\to[0,1]$ be a uniformly continuous and nonincreasing function satisfying $u_{0}(x)=1$ for $x\leq 0$ and $u_{0}(x)=\tilde{\theta}_{I}$ for $x\geq x_{0}$, where $x_{0}>0$ is fixed. Clearly, $u(t_{0},\cdot+X^{+}_{\tilde{\theta}_{I}}(t_{0}))\leq u_{0}$. Applying comparison principle and Lemma \ref{lem-ignition-property}, we find
\begin{equation*}
u(t,x+X^{+}_{\tilde{\theta}_{I}}(t_{0}))\leq u_{I}(t-t_{0},x;u_{0})\leq\phi_{I}(x-c_{I}(t-t_{0})-\xi_{I})+\ep_{I}e^{-\om_{I}(t-t_{0})},\quad x\in\R,\quad t\geq t_{0}.
\end{equation*}
Let $\xi_{I}(\frac{\la}{2})$ be the unique point such that $\phi_{I}(\xi_{I}(\frac{\la}{2}))=\frac{\la}{2}$ and $T>0$ be such that $\ep_{I}e^{-\om_{I}T}=\frac{\la}{2}$ (we may make $\ep_{I}>\frac{\la}{2}$ if necessary). Setting $x_{*}=c_{I}(t-t_{0})+\xi_{I}+\xi_{I}(\frac{\la}{2})$, we conclude from the monotonicity of $\phi_{I}$ that for $x\geq x_{*}+1$ and $t\geq t_{0}+T$,
\begin{equation*}
\begin{split}
u(t,x+X^{+}_{\tilde{\theta}_{I}}(t_{0}))&\leq\phi_{I}(x_{*}+1-c_{I}(t-t_{0})-\xi_{I})+\ep_{I}e^{-\om_{I}T}\\
&<\phi_{I}(x_{*}-c_{I}(t-t_{0})-\xi_{I})+\ep_{I}e^{-\om_{I}T}=\la.
\end{split}
\end{equation*}
It then follows from the definition of $X^{+}(t)$ that
\begin{equation*}
X^{+}(t)\leq x_{*}+1+X^{+}_{\tilde{\theta}_{I}}(t_{0})=c_{I}(t-t_{0})+\xi_{I}+\xi_{I}(\frac{\la}{2})+1+X^{+}_{\tilde{\theta}_{I}}(t_{0}),\quad t\geq t_{0}+T.
\end{equation*}
Setting $C_{*}:=\sup_{t_{0}\in\R}|X^{+}(t_{0})-X^{+}_{\tilde{\theta}_{I}}(t_{0})|<\infty$ due to Lemma \ref{lem-bounded-interface-width}, we conclude
\begin{equation*}
X^{+}(t)-X^{+}(t_{0})\leq c_{I}(t-t_{0})+\xi_{I}+\xi_{I}(\frac{\la}{2})+1+C_{*},\quad t\geq t_{0}+T.
\end{equation*}

It remains to show that
\begin{equation}\label{what-to-show}
X^{+}(t)-X^{+}(t_{0})\leq\xi_{*},\quad t\in[t_{0},t_{0}+T]
\end{equation}
for some $\xi_{*}>0$ independent of $t_{0}$. To do so, let $\tilde{u}_{0}$ be the $u_{0}$ in the proof of Theorem \ref{lem-propagation-estimate}. Then, we have $\tilde{u}_{0}(\cdot-X^{-}(t_{0}))\leq u(t_{0},\cdot)\leq u_{0}(\cdot-X^{+}_{\tilde{\theta}_{I}}(t_{0}))$, where $X^{-}(t)=X^{-}_{\la}(t)$. Since $f_{B}\leq f\leq f_{I}$, we apply comparison principle to conclude that
\begin{equation*}
u_{B}(t-t_{0},x-X^{-}(t_{0});\tilde{u}_{0})\leq u(t,x)\leq u_{I}(t-t_{0},x-X^{+}_{\tilde{\theta}_{I}}(t_{0});u_{0}),\quad x\in\R,\quad t\geq t_{0}.
\end{equation*}
We then conclude \eqref{what-to-show} from the continuity of $u_{B}(t-t_{0},x-X^{-}(t_{0})$ and $u_{I}(t-t_{0},x-X^{+}_{\tilde{\theta}_{I}}(t_{0});u_{0})$, and the fact $\sup_{t_{0}\in\R}|X^{-}(t_{0})-X^{+}_{\tilde{\theta}_{I}}(t_{0})|<\infty$ due to Lemma \ref{lem-bounded-interface-width}. This completes the proof.
\end{proof}

Finally, we prove Corollary \ref{cor-modified-interface-ii}.

\begin{proof}[Proof of Corollary \ref{cor-modified-interface-ii}]
Note first we can find a $C^{2}$ Fisher-KPP nonlinearity $f_{\rm KPP}:[0,1]\to\R$ such that $f(t,x,u)\leq f_{\rm KPP}(u)$ for all $(t,x,u)\in\R\times\R\times[0,1]$. Let $u(t,x)$ be the transition front as in the statement of Corollary \ref{cor-modified-interface-ii}, that is, there exist $r>0$ and $h>0$ such that
\begin{equation*}
u(t_{0},x+X(t_{0}))\leq e^{-r(x-h)},\quad (t_{0},x)\in\R\times\R,
\end{equation*}

Fixed $\la\in(0,1)$. Setting $h_{0}:=h+\sup_{t_{0}\in\R}|X(t_{0})-X^{+}_{\la}(t_{0})|<\infty$, we find
\begin{equation*}
u(t_{0},x+X^{+}_{\la}(t_{0})+h_{0})\leq e^{-rx},\quad (t_{0},x)\in\R\times\R.
\end{equation*}
Then, we can find some uniformly continuous function $u_{0}:\R\to[0,1]$ satisfying
\begin{equation*}
\lim_{x\to-\infty}u_{0}(x)=1\quad\text{and}\quad\lim_{x\to\infty}\frac{u_{0}(x)}{e^{-rx}}=1
\end{equation*}
such that
\begin{equation*}
u(t_{0},x+X^{+}_{\la}(t_{0})+h_{0})\leq u_{0}(x),\quad (t_{0},x)\in\R\times\R.
\end{equation*}
Note that we may assume, without loss of generality, that $r$ is so small that it is the decay rate of some traveling wave of $\phi_{r}(x-c_{r}t)$ (satisfying $\phi_{r}(-\infty)=1$ and $\phi_{r}(\infty)=0$) with speed $c_{r}=\frac{\int_{\R}J(y)e^{ry}dy-1+f_{\rm KPP}'(0)}{r}>0$ of
\begin{equation}\label{eqn-kpp-homo}
u_{t}=J\ast u-u+f_{\rm KPP}(u),
\end{equation}
that is, $\lim_{x\to\infty}\frac{\phi_{r}(x)}{e^{-rx}}=1$ (see \cite{CaCh04} and \cite{ShZh12-2}). In particular, we have
\begin{equation}\label{exact-decay-rate-2015201}
\lim_{x\to\infty}\frac{u_{0}(x)}{\phi_{r}(x)}=1.
\end{equation}
Moreover, there holds
\begin{equation}\label{exact-decay-rate-20152015}
\lim_{x\to\infty}\frac{\phi_{r}'(x)}{\phi_{r}(x)}=-r.
\end{equation}
To see this, we notice $\frac{J\ast\phi_{r}}{\phi_{r}}-1+c_{r}\frac{\phi_{r}'}{\phi_{r}}+\frac{f_{\rm KPP}(\phi_{r})}{\phi_{r}}=0$. Clearly, $\lim_{x\to\infty}\frac{f_{\rm KPP}(\phi_{r}(x))}{\phi_{r}(x)}=f_{\rm KPP}'(0)$. For $\frac{J\ast\phi_{r}}{\phi_{r}}$, we have
\begin{equation*}
\frac{[J\ast\phi_{r}](x)}{\phi_{r}(x)}=\frac{e^{-rx}}{\phi_{r}(x)}\int_{\R}J(y)e^{ry}\frac{\phi_{r}(x-y)}{e^{-r(x-y)}}dy\to J(y)e^{ry}dy\quad\text{as}\quad x\to\infty
\end{equation*}
by \eqref{exact-decay-rate-2015201} and dominated convergence theorem. From which, we conclude \eqref{exact-decay-rate-20152015}.

Then, arguing as in the proof of Corollary \ref{cor-modified-interface}, we conclude the result from the stability of $\phi_{r}(x-c_{r}t)$, that is,
\begin{equation}\label{stability-kpp}
\lim_{t\to\infty}\bigg|\frac{u_{\rm KPP}(t,x;u_{0})}{\phi_{r}(x-c_{r}t)}-1\bigg|=0,
\end{equation}
where $u_{\rm KPP}(t,x;u_{0})$ is the solution of \eqref{eqn-kpp-homo} with initial data $u_{\rm KPP}(0,\cdot;u_{0})=u_{0}$. We remark that \eqref{stability-kpp} follows from \cite[Theorem 2.6]{ShZh12-2}. Also, by means of \eqref{exact-decay-rate-2015201} and \eqref{exact-decay-rate-20152015}, it can be proven as that of \cite[Theorem 1.3]{ShSh14-kpp}.
\end{proof}


\section*{Acknowledgements} 

The authors would like to thank the referee for carefully reading the manuscript, pointing out some problems that we were not aware of, and drawing our attention to the reference \cite{BaChm99}.

\appendix

\section{Ignition traveling waves}\label{app-ig-tw}

Consider the homogeneous ignition equation
\begin{equation}\label{eqn-homo-ignition}
u_{t}=J\ast u-u+f_{I}(u),\quad (t,x)\in\R\times\R,
\end{equation}
where $J$ is as in (H1), and the $C^{2}$ function $f_{I}:[0,1]\to\R$ is of standard ignition type, that is, there is $\theta_{I}\in(0,1)$ such that
\begin{equation*}
f_{I}(u)=0,\,\,u\in[0,\theta_{I}]\cup\{1\},\quad f_{I}(u)>0,\,\,u\in(\theta_{I},1)\quad\text{and}\quad f_{I}'(1)<0.
\end{equation*}

It was proven in \cite{Cov-thesis} that the problem
\begin{equation*}
\begin{cases}
J\ast\phi-\phi+c\phi_{x}+f_{I}(\phi)=0,\\
\phi_{x}<0,\,\,\phi(0)=\theta_{I},\,\,\phi(-\infty)=1\,\,\text{and}\,\,\phi(\infty)=0.
\end{cases}
\end{equation*}
for $(c,\phi)$ has a unique classical solution $(c_{I},\phi_{I})$ with $c_{I}>0$.

We used the following result in the previous sections.

\begin{lem}\label{lem-ignition-property}
Let $u_{0}:\R\to[0,1]$ be uniformly continuous and satisfy
\begin{equation*}
\lim_{x\to-\infty}u_{0}(x)=1\quad\text{and}\quad u_{0}(x)\leq e^{-\al_{0} (x-x_{0})},\,\,x\in\R
\end{equation*}
for some $\al_{0}>0$ and $x_{0}\in\R$, then there exist $\om_{I}=\om_{I}(\al_{0})>0$ and $\ep_{I}>0$ such that for any $\ep\in(0,\ep_{I}]$ there exist $\xi^{\pm}_{I}=\xi^{\pm}_{I}(\ep,u_{0})\in\R$ such that
\begin{equation*}
\phi_{I}(x-c_{I}t-\xi^{-}_{I})-\ep e^{-\om_{I}t}\leq u_{I}(t,x;u_{0})\leq\phi_{I}(x-c_{I}t-\xi^{+}_{I})+\ep e^{-\om_{I}t},\quad x\in\R
\end{equation*}
for all $t\geq0$, where $u_{I}(t,x;u_{0})$ is the solution of \eqref{eqn-homo-ignition} with initial data $u_{I}(0,\cdot;u_{0})=u_{0}$.
\end{lem}

Lemma \ref{lem-ignition-property} can be proven as \cite[Theorem 1.4]{ShSh14-3}; we here simply recall it for completeness. To do so, we fix $L_{1}>0$ so large that
\begin{equation*}
\phi_{I}(-L_{1})\geq\frac{1+\tilde{\theta}_{I}}{2}\quad\text{and}\quad\phi_{I}(L_{1})\leq\frac{\theta_{I}}{2},
\end{equation*}
where $\tilde{\theta}_{I}\in(\theta_{I},1)$ is such that
\begin{equation}\label{uniform-decaying-1-ig}
f_{I}'(u)\leq-\tilde{\beta}_{I},\quad u\in[\tilde{\theta}_{I},1]
\end{equation}
for some $\tilde{\beta}_{I}>0$ (such $\tilde{\theta}_{I}$ and $\tilde{\beta}_{I}$ exist due to $f_{I}'(1)<0$).

For $\al>0$, let $\Ga_{\al}:\R\to[0,1]$ be a smooth nonincreasing function satisfying
\begin{equation*}
\Ga_{\al}=\begin{cases}
1,\quad &x\leq-L_{1}-1,\\
e^{-\al(x-L_{1})},\quad &x\geq L_{1}+1.
\end{cases}
\end{equation*}

We have

\begin{lem}\label{lem-tech}
There exists $\al_{*}>0$ such that such that for any $\al\in(0,\al_{*}]$ there exists $L_{2}=L_{2}(\al)>L_{1}+1$ such that
\begin{equation*}
\big|[J\ast\Ga_{\al}](x)-e^{-\al(x-L_{1})}\big|\leq\frac{c_{I}}{4}\al e^{-\al(x-L_{1})},\quad x\geq L_{2}.
\end{equation*}
\end{lem}
\begin{proof}
See \cite[Lemma 4.1]{ShSh14-3}. It requires the symmetry of $J$ so that $\int_{\R}J(x)e^{\al x}dx-1=O(\al^{2})$.
\end{proof}

Now, we prove Lemma \ref{lem-ignition-property}.

\begin{proof}[Proof of Lemma \ref{lem-ignition-property}]

Let $\al=\frac{\al_{0}}{2}$ and $L_{2}=L_{2}(\al)$ as in Lemma \ref{lem-tech}. For any $\ep\in(0,\ep_{I}]$, where $\ep_{I}>0$ is to be chosen, we can find $\xi^{\pm}=\xi_{0}^{\pm}(\ep,u_{0})\in\R$ such that
\begin{equation}\label{ignition-homo-initial-estimate}
\phi_{I}(x-\xi^{-})-\ep\Ga_{\al}(x-\xi^{-})\leq u_{0}(x)\leq\phi_{I}(x-\xi^{+})+\ep\Ga_{\al}(x-\xi^{+}),\quad x\in\R.
\end{equation}
Setting $\xi^{\pm}(t)=\xi^{\pm}\pm\frac{A\ep}{\om}(1-e^{-\om t})$, where $A>0$ and $\om>0$ is to be chosen, we define
\begin{equation*}
u^{\pm}(t,x)=\phi(x-ct-\xi^{\pm}(t))\pm\ep e^{-\om t}\Ga(x-ct-\xi^{\pm}(t)),\quad t\geq0,
\end{equation*}
where $\phi=\phi_{I}$, $c=c_{I}$ and $\Ga=\Ga_{\al}$. Clearly, $u^{-}(0,\cdot)\leq u_{0}\leq u^{+}(0,\cdot)$. Thus, if we can show that $u^{-}(t,x)$ and $u^{+}(t,x)$ are sub- and super-solutions, respectively, then the lemma follows.

We show that $u^{-}(t,x)$ is a sub-solution; $u^{+}(t,x)$ being a super-solution can be proven along the same line. We compute
\begin{equation*}
\begin{split}
&u^{-}_{t}-[J\ast u^{-}-u^{-}]-f_{I}(u^{-})\\
&=A\ep e^{-\om t}\phi'+\ep\om e^{-\om t}\Ga-\ep e^{-\om t}(A\ep e^{-\om t}-c)\Ga'+\ep e^{-\om t}[J\ast\Ga-\Ga]+f_{I}(\phi)-f_{I}(u^{-}),
\end{split}
\end{equation*}
where $\phi$, $\phi'$, $\Ga$ and $\Ga'$ are computed at $x-ct-\xi^{-}(t)$ and $J\ast\Ga=\int_{\R}J(x-y)\Ga(y-ct-\xi^{-}(t))dy$. We consider three cases.

\paragraph{\textbf{Case 1.} $x-ct-\xi^{-}(t)\leq-L_{1}-1$} In this case, $\Ga=1$, $\Ga'=0$ and hence $J\ast\Ga-\Ga\leq1-1=0$. Moreover, $\phi\geq\frac{1+\tilde{\theta}_{I}}{2}$ by the monotonicity of $\phi$ and the choice of $L_{1}$, which implies that $u_{-}\geq\phi-\ep_{I}\geq\tilde{\theta}_{I}$ if we choose
\begin{equation}\label{condition-ep-1}
\ep_{I}\leq\frac{1-\tilde{\theta}_{I}}{2}.
\end{equation}
It then follows that $f_{I}(\phi)-f_{I}(u^{-})\leq-\ep\tilde{\beta}_{I}e^{-\om t}\Ga$. Hence, we obtain
\begin{equation*}
u^{-}_{t}-[J\ast u^{-}-u^{-}]-f_{I}(u^{-})\leq\om\ep e^{-\om t}\Ga-\ep\tilde{\beta}_{I}e^{-\om t}\Ga\leq0
\end{equation*}
if we choose
\begin{equation}\label{condition-om-1}
\om\leq\tilde{\beta}_{I}.
\end{equation}

\paragraph{\textbf{Case 2.} $x-ct-\xi^{-}(t)\in[-L_{1}-1,L_{2}]$} In this case,
\begin{equation*}
A\ep e^{-\om t}\phi'\leq A\ep e^{-\om t}\sup_{x\in[-L_{1}-1,L_{2}]}\phi'(x)<0,
\end{equation*}
\begin{equation*}
\ep\om e^{-\om t}\Ga-\ep e^{-\om t}(A\ep e^{-\om t}-c)\Ga'+\ep e^{-\om t}[J\ast\Ga-\Ga]\leq\ep e^{-\om t}(\om+1)
\end{equation*}
if we choose
\begin{equation}\label{condition-ep-2}
\ep_{I}\leq\frac{c}{A},
\end{equation}
and $f_{I}(\phi)-f_{I}(u^{-})\leq(\sup_{u\in[0,2]}|f'_{I}(u)|)\ep e^{-\om t}$ (note that it's safe to extend $f_{I}$ to $(1,2]$ so that $\sup_{u\in[0,2]}|f'_{I}(u)|<\infty$). It then follows that
\begin{equation*}
u^{-}_{t}-[J\ast u^{-}-u^{-}]-f_{I}(u^{-})\leq \ep e^{-\om t}\bigg[A\sup_{x\in[-L_{1}-1,L_{2}]}\phi'(x)+\om+1+\sup_{u\in[0,2]}|f'_{I}(u)|\bigg]\leq0
\end{equation*}
if we choose
\begin{equation}\label{condition-A}
A\geq-\bigg[\sup_{x\in[-L_{1}-1,L_{2}]}\phi'(x)\bigg]^{-1}\bigg[1+2\sup_{u\in[0,2]}|f'_{I}(u)|\bigg],
\end{equation}
since $\om\leq\tilde{\beta}_{I}\leq\sup_{u\in[0,2]}|f'_{I}(u)|$ due to \eqref{condition-om-1}.

\paragraph{\textbf{Case 3.} $x-ct-\xi^{-}(t)\geq L_{2}$} In this case, $\Ga=e^{-\al(x-ct-\xi^{-}(t)-L_{1})}$, $\Ga'=-\al\Ga$ and hence,
\begin{equation*}
\ep\om e^{-\om t}\Ga-\ep e^{-\om t}(A\ep e^{-\om t}-c)\Ga'=\ep e^{-\om t}[\om+A\al\ep e^{-\om t}-c\al]\Ga.
\end{equation*}
By Lemma \ref{lem-tech}, we have $\ep e^{-\om t}[J\ast\Ga-\Ga]\leq\ep e^{-\om t}\frac{\al c}{4}\Ga$. Since $f_{I}(\phi)=0=f_{I}(u^{-})$ (note it's safe to do zero extension of $f$ on $(-\infty,0)$), we obtain
\begin{equation*}
u^{-}_{t}-[J\ast u^{-}-u^{-}]-f_{I}(u^{-})\leq\ep e^{-\om t}\bigg[\om+A\al\ep e^{-\om t}-c\al+\frac{\al c}{4}\bigg]\Ga\leq0
\end{equation*}
if we choose
\begin{equation}\label{condition-om-ep}
\om\leq\frac{\al c}{4}\quad\text{and}\quad \ep_{I}\leq\frac{c}{4A}
\end{equation}

Consequently, if we choose $A$ as in \eqref{condition-A}, $\om$ as in \eqref{condition-om-1} and \eqref{condition-om-ep}, and $\ep_{I}$ as in \eqref{condition-ep-1} and \eqref{condition-om-ep}, then we have $u^{-}_{t}-[J\ast u^{-}-u^{-}]-f_{I}(u^{-})\leq0$ for $t\geq0$. This completes the proof.
\end{proof}


\bibliographystyle{amsplain}

\end{document}